\documentclass{article}
\usepackage{authblk}
\usepackage[utf8]{inputenc}
\usepackage{graphicx}
\usepackage{amsmath, amsthm, amssymb}
\usepackage{amsfonts}
\usepackage{amsthm}
\usepackage{enumitem}
\usepackage{subcaption}
\usepackage{comment}
\usepackage{mathdots}
\usepackage{enumitem}
\usepackage{times}
\usepackage{mathrsfs}

\newtheorem{theorem}{Theorem}[section]

\newtheorem{corollary}[theorem]{Corollary}
\newtheorem{lemma}[theorem]{Lemma}

\theoremstyle{definition}
\newtheorem{example}[theorem]{Example}
\newtheorem{remark}[theorem]{Remark}
\newtheorem{definition}[theorem]{Definition}

\usepackage{xfrac}
\usepackage{lipsum}
\usepackage{blindtext}
\usepackage{mathtools}
\usepackage{array}
\usepackage{titlesec}
\usepackage{multirow} 
\usepackage{booktabs}
\usepackage{tabularx}
\usepackage{makecell}
\usepackage{flafter}
\newcommand{\R}{\mathbb{R}}
\newcommand{\mc}{\mathcal}
\newcommand{\N}{\mathbb{N}}
 
\newcommand{\F}{\mathcal{F}}

\newcommand{\ii}{\mathrm{i}}

\DeclareMathOperator{\P2}{{\mathscr{P}}_{\text{\emph{\tiny{II }}}}}

\newcommand{\vecm}{\operatorname{vec}}
\newcommand{\T}{\mathscr{T}}

\newcommand{\PIII}{\mathscr{P}_{\text{\emph{\tiny{III}}}}}
\newcommand{\PII}{\mathscr{P}_{\text{\emph{\tiny{II}}}}}

\newcommand{\npmatrix}[1]{\left( \begin{matrix} #1 \end{matrix} \right)}

\title{Powers of Karpelevi\v c arcs and their Sparsest Realising matrices}

\author[1]{Priyanka Joshi}
\author[2]{Stephen Kirkland}
\author[1]{Helena \v Smigoc}
\affil[1]{School of Mathematics and Statistics, University College Dublin, Belfield, Dublin 4, Ireland}
\affil[2]{Department of Mathematics, University of Manitoba, Winnipeg, MB, Canada}

\begin{document}
\maketitle

\begin{abstract}
The region in the complex plane containing the eigenvalues of all $n\times n$ stochastic matrices was described by Karpelevi\v c in 1988, and it is since then known as the Karpelevi\v c region. The boundary of the Karpelevi\v c region is the union of disjoint arcs called the Karpelevi\v c arcs. We provide a complete characterization of the Karpelevi\v c arcs that are powers of some other Karpelevi\v c arc. Furthermore, we find the necessary and sufficient conditions for a sparsest stochastic matrix associated with the Karpelevi\v c arc of order $n$ to be a power of another stochastic matrix. 
\end{abstract}

\section{Introduction}

A square entrywise nonnegative matrix is called stochastic if each row sum equals $1$. Stochastic matrices and their properties are central to 
the study of Markov chains; in particular, the eigenvalues of a stochastic matrix govern the long-term behavior of the iterates of the corresponding Markov chain. Consequently, there is a long-standing interest in localising the eigenvalues of stochastic matrices, and a classic problem of Kolmolgorov \cite{kolmogorov1937markov} asks for a description of the region in the complex plane containing all the eigenvalues of all $n\times n$ stochastic matrices.

That region, 
denoted by $\Theta_n$, was characterised by Karpelevi\v c \cite{karpelevic}. He showed that the boundary of the Karpelevi\v c region, denoted by $\partial \Theta_n$, is a union of disjoint arcs called the Karpelevi\v c arcs. Kirkland, Laffey and \v Smigoc \cite{jmaa} expanded on this result by identifying the point on the boundary of the Karpelevi\v c region with a given argument $\theta \in [0,2\pi)$.

The problem of determining a stochastic $p$-th root of a stochastic matrix finds its motivation in the theory of Markov Chains, where it corresponds to the problem of finding a transition matrix over a shorter time interval from a given transition matrix. This problem was considered for example in \cite{MR2794585}.

Johnson and Paparella \cite{johnsonpaperella} posed a conjecture that selected Karpelevi\v c arcs are powers of some other Karpelevi\v c arcs. Kim and Kim \cite{kimkim} proved their conjecture. However, results in \cite{johnsonpaperella} and \cite{kimkim} only partially answer the question of characterising the powers of the Karpelevi\v c arcs. After establishing notation and recalling the necessary background results on the Karpelevi\v c region in Section \ref{sec:backround1}, we give a complete characterization of the Karpelevi\v c arcs that can be written as a power of another Karpelevi\v c arc in Section \ref{sec:ArcPowers}. 

 Johnson and Paparella \cite{johnsonpaperella} also considered the question of constructing stochastic matrices realising the boundary of the Karpelevi\v c region. For each Karpelevi\v c arc, they provided a single parametric family of stochastic matrices that realises eigenvalues on that arc.  Kirkland and \v Smigoc in \cite{kirklandsmigoc} described all sparsest $n\times n$ stochastic matrices realising eigenvalues on the border of the Karpelevi\v c region.  In Section \ref{sec:background2} we recall the results from  \cite{kirklandsmigoc}, and establish notation for digraphs and the associated stochastic matrices. This background is needed in Section \ref{sec:powers_of_matrices}, where we
 characterise the sparsest realising matrices that can be written as a power of another stochastic matrix.
 
This paper is organised in two parts that can be read independently. Section \ref{sec:backround1} provides the background to Section \ref{sec:ArcPowers} where the complete characterisation of the Karpelevi\v c arcs that can be written as a power of another Karpelevi\v c arc is given in Theorem \ref{thm:powers} and Corollary \ref{cor:powers1}. A reader that is more interested in the powers of matrices and results developed in  Section \ref{sec:powers_of_matrices}, could view  Theorem \ref{thm:powers} and Corollary \ref{cor:powers1} as part of background results, and learn about additional background in Section \ref{sec:background2}.  

\section{Background and Notation on the Karpelevi\v c Region}\label{sec:backround1}

Given $n \in \mathbb{N}$, the set $$\F_n=\{\sfrac{p}{q} \mid 0 \leq p < q \leq n, \gcd(p,q)=1\}$$ is called \emph{the set of Farey fractions of order $n$}.
The pair $(\sfrac{p}{q},\sfrac{r}{s})$ is called \emph{a Farey pair (of order $n$)}, if $\sfrac{p}{q},\sfrac{r}{s} \in \mc F_n$, $\sfrac{p}{q} < \sfrac{r}{s}$ and $\sfrac{p}{q}< x <\sfrac{r}{s}$ implies $x \not \in \mc F_n$. The Farey fractions $\sfrac{p}{q}$ and $\sfrac{r}{s}$  are called \emph{Farey neighbours} if one of  $(\sfrac{p}{q},\sfrac{r}{s})$ and $(\sfrac{r}{s}, \sfrac{p}{q})$ is a Farey pair. It is well known that the Farey fractions $\sfrac{p}{q}$ and $\sfrac{r}{s}$ form a Farey pair iff $q+s>n$ and $\left| qr-ps\right|=1 $. 

For $n$, $q$ and $s$, satisfying, $q<s\leq n$, $\gcd(q,s)=1$, and $q+s>n$, there exits precisely two Farey pairs in $\mc F_n$ with denominators $q$ and $s$: 
\begin{equation}\label{eq:conjugate farey pairs}
    \mc F_n(q,s):=\{(\sfrac{p}{q},\sfrac{r}{s}),(\sfrac{(s-r)}{s},\sfrac{(q-p)}{q})\}.
\end{equation}
Hence, there exist unique $p$ and $r$ so that $(\sfrac{p}{q},\sfrac{r}{s})\in \mc F_n(q,s)$. We will denote those as $p(q,s)$ and $r(q,s)$, and when clear from the context just as $p$ and $r$.

Further parameters associated with  $\mc F_n(q,s)$ that were first defined in \cite{jmaa} and will also be needed in this work, are given below:
\begin{align*}
d(q,s)&=\left \lfloor \frac{n}{q}\right \rfloor,\\
\delta(q,s)&=\gcd(d(q,s),s), \\
s&=s_1(q,s) \delta(q,s), \\ 
d(q,s)&=d_1(q,s) \delta(q,s). 
\end{align*}
Note that $d(q,s)$, $\delta(q,s)$, $d_1(q,s)$, $s_1(q,s)$ depend on $n$ as well as on $q$ and $s$. Our notation does not capture this dependence, as we will always fix $n$. Furthermore, once $q$ and $s$ are established, we will abbreviate notation to $d$, $\delta$, $s_1$ and $d_1$. 

\begin{example}\label{ex:parameters}
Let $q=3$ and $s=14$. Then $\mc F_{n}(3,14)=\{(\sfrac{1}{3},\sfrac{5}{14}),(\sfrac{9}{14},\sfrac{2}{3})\}$ for $n \in \{14,15,16\}$. For $n=14$ we get $d(3,14)=\left \lfloor \frac{14}{3}\right \rfloor=4$, $\delta(3,14)=2, d_1(3,14)=2$ and $s_1(3,14)=7$. On the other hand, for $n=15$ and for $n=16$ we get $d(3,14)=\left \lfloor \frac{16}{3}\right \rfloor=5$, $\delta(3,14)=1, d_1(3,14)=5$ and $s_1(3,14)=14.$
\end{example}

The Karpelevi\v c region was first described in \cite{karpelevic}. We recall a version of this result from  \cite{ito} below. 

\begin{theorem}(\cite{karpelevic},\cite{ito})\label{thm:Karpelevic}
The region $\Theta_n$ is symmetric with respect to the real axis, is included in the unit disc $\{z\in C|\hspace{1mm} |z|\leq 1\}$, and intersects the unit circle $\{z\in C| \hspace{1mm} |z|=1\}$ at the points $\{e^{\frac{2\pi i p}{q}}|\hspace{1mm} p/q \in \mc F_n\}$. The boundary of $\Theta_n$ consists of these points and curvilinear arcs connecting them in circular order.

The arc with endpoints $e^{\frac{2\pi i p}{q}}$ and $e^{\frac{2\pi i r}{s}}$, $q <  s$, is given by the following parametric equation:
\begin{equation}\label{eq:karpelevic}
    t^s(t^q-1+\alpha)^{\lfloor \frac{n}{q}\rfloor}=\alpha^{\lfloor \frac{n}{q}\rfloor}t^{q\lfloor \frac{n}{q}\rfloor}, \alpha \in [0,1].
\end{equation}
\end{theorem}

For a Farey pair associated with the endpoints of an arc, the above theorem assumes $q$ to be the smallest denominator of the two fractions appearing in the pair. 
 As equation \eqref{eq:karpelevic} depends on $q$ and $s$ only, the same equation describes both arcs associated with the Farey pairs in $\mc F_n(q,s)$. These two arcs are conjugates of each other, and we will (when convenient) only concentrate on one of them, say the arc associated with the Farey pair, $(\sfrac{p}{q}, \sfrac{r}{s})$ with $q<s$. 

For $q<s$, $\gcd(q,s)=1$ and $q+s>n$, we set $p:=p(q,s)$ and $r:=r(q,s)$, and introduce the following notation:
\begin{enumerate}
\item $\arg(q,s)\equiv (\sfrac{2 \pi p}{q},\sfrac{2 \pi r}{s}) \cup (\sfrac{2 \pi  (s-r)}{s}, \sfrac{2 \pi  (q-p)}{q})$,
\item ${\mc K}_n(q,s)\equiv \partial \Theta_n \cap  \{z \in \mathbb{C} \mid \arg(z) \in \arg(q,s)\}$,
\item $\mc K_n$ the set of all $\mc K_n(q,s)$ for a fixed $n \in \mathbb{N}$. 
\end{enumerate}
Theorem \ref{thm:Karpelevic} tells us that $\mc K_n(q,s)$ contains a pair of arcs that are complex conjugates of each other, and whose points satisfy the parametric equation \eqref{eq:karpelevic}. For the Farey pair, $(\sfrac{p}{q},\sfrac{r}{s}) \in \mathcal F_n(q,s)$ with $q<s$, we can substitute $d=\lfloor \frac{n}{q}\rfloor$ in equation \eqref{eq:karpelevic}. The  polynomial $f_{\alpha}(t)= t^s(t^q-\beta)^d-\alpha^d t^{qd}, \alpha \in [0,1]$, are called the \emph{Ito polynomials} for the Farey pair $(\sfrac{p}{q},\sfrac{r}{s}),q<s$. The \emph{reduced Ito polynomials} are the polynomials obtained from the Ito polynomials by removing the zero roots.

\begin{example}\label{ex:karpelevic_thm}
In Example \ref{ex:parameters} we have seen that $\mc F_{n}(3,14)=\{(\sfrac{1}{3},\sfrac{5}{14}),(\sfrac{9}{14},\sfrac{2}{3})\}$ for $n \in \{14,15,16\}$. In all three cases $\arg(3,14)= (\sfrac{2 \pi }{3},\sfrac{5 \pi }{7}) \cup (\sfrac{9 \pi }{7}, \sfrac{4 \pi}{3})$, and $\mc K_{n}(3,14)=\partial \Theta_{n} \cap \{z \in \mathbb{C} \mid \arg(z) \in \arg(3,14) \}$. 
\end{example}

Theorem \ref{thm:Karpelevic} describes $\partial \Theta_n$. However, from the theorem, it is not immediately clear, given $\theta \in [0,2\pi]$, how to find  $\rho$ satisfying $\rho e^{i \theta} \in \partial \Theta_n$. This issue was resolved in Theorem 1.2 from \cite{jmaa} which is restated below. 

\begin{definition}
For $n \geq 2$, we define $\rho_n: [0,2\pi] \rightarrow \R_+$ to be the positive number satisfying $\rho_n(\theta)e^{i\theta} \in \partial \Theta_n$. \end{definition}

\begin{theorem}[Theorem 1.2, \cite{jmaa}]\label{thm:gcd200}
Let $(\frac{p}{q},\frac{r}{s}) \in \F_n(q,s)$, $q <s$, and $\theta \in [\sfrac{2 \pi p}{q},\sfrac{2 \pi r}{s}]$. Then $\rho_n(\theta)=\mu^{d_1}$,  where $\mu$ is the unique positive solution to 
\begin{equation}\label{eq:implicit}
    \mu^{s_1} \sin(q\theta)- \mu^{qd_1} \sin \left(\frac{s_1}{d_1}\theta  -\frac{2 \pi r}{d} \right)-\sin\left((q-\frac{s_1}{d_1})\theta+\frac{2 \pi r}{d} \right)=0.
    \end{equation}
Furthermore,  $\rho_n(\theta) e^{\ii \theta}$ is a root of the rational function $\phi_{\alpha}(t) = (t^q-\beta)^d-\alpha^dt^{qd-s}$,  where $\alpha$ is given by $$\alpha \sin\left((q-\frac{s_1}{d_1})\theta+\frac{2 \pi r}{d} \right)=\mu^{s_1}  \sin(q\theta).$$ 
\end{theorem}

\begin{remark}\label{remark:gcd200}
To apply Theorem \ref{thm:gcd200} directly to $(\frac{p}{q},\frac{r}{s}) \in \F_n$, the assumption $q <s$ is necessary. To use the theorem for $\theta' \in (\frac{2 \pi(s-r)}{s},\frac{2 \pi(q-p)}{q}) \in \F_n(q,s)$ we note that $\rho_n e^{i \theta'} \in \partial\Theta_n$ if and only if $\rho_n e^{i \theta} \in \partial\Theta_n$ for $\theta=2 \pi-\theta'$. So for $\theta' \in (\frac{2 \pi(s-r)}{s},\frac{2 \pi(q-p)}{q}) \in \F_n$, we have $\mu^{d_1} e^{i\theta'} \in \mc K_n(q,s)$ if and only if $\mu$  satisfies:
\begin{equation}\label{eq:implicit-conjugate}
     \mu^{s_1} \sin(q\theta')- \mu^{qd_1} \sin \left(\frac{s_1}{d_1}\theta'  +\frac{2 \pi (r-s)}{d} \right)-\sin\left((q-\frac{s_1}{d_1})\theta'-\frac{2 \pi (r-s)}{d} \right)=0.
\end{equation}
This equation can be obtained from \eqref{eq:implicit} by inserting $\theta=2 \pi-\theta'$. 
 \end{remark}

\begin{example}
To determine $\rho_{14}(\sfrac{29 \pi}{42})$ we first note that $\sfrac{29 \pi}{42} \in [\sfrac{2 \pi }{3},\sfrac{5 \pi }{7}]$ and $(\sfrac{1}{3},\sfrac{5}{14}) \in \mathcal F_{14}(3,14)$. Substituting $d=4, \delta=2, d_1=2, s_1=7$ and $\theta=\sfrac{29 \pi}{42}$ in \eqref{eq:implicit}, we get: 
$$\mu^{7} \sin(\sfrac{29 \pi}{14})+ \mu^{6} \sin (\sfrac{\pi}{12})-\sin\left(\sfrac{181 \pi}{84} \right)=0,$$
which has a unique positive solution $\mu_0$ (approximately equal to $0.99542$). Thus, $\rho_{14}(\sfrac{29 \pi}{42})=\mu_0^2$.
 
 Now, for $(\sfrac{9}{14},\sfrac{2}{3}) \in \mc F_{14}(3,14)$ let $\theta'=2\pi-\theta=\sfrac{55\pi}{42}$. We see that $\sfrac{55\pi}{42} \in[\sfrac{9 \pi}{7},\sfrac{4 \pi}{3}]$. Substituting $\theta'=55\pi/42$ in equation \eqref{eq:implicit-conjugate}, we get:
 
 $$\mu^{7} \sin(\sfrac{55 \pi}{14})- \mu^{6} \sin (\sfrac{\pi}{12})-\sin\left(\sfrac{323 \pi}{84} \right)=0,$$ which again has the unique positive solution $\mu_0$. Hence, $\rho_{14}(55\pi/42)=\mu_0^2$.
\end{example}

\section{Powers of Karpelevi\v c Arcs}\label{sec:ArcPowers}

For $\mc S \subset \mathbb{C}$ and $c \in \N$, we define the $c$-th power of $\mc S$ to be:
$$\mc S^c:=\{\lambda^c| \lambda \in\mc S\}.$$
In this section, we aim to understand, when one Karpelevi\v c arc is a power of another Karpelevi\v c arc. In other words, we want to identify $q,s,\hat q, \hat s, n$ and $c$ so that $\mc K_n(q,s)=\mc K_n(\hat q,\hat s)^c$. 

We start by considering the necessary conditions on the associated Farey pairs that assure that the endpoints of an arc are mapped to the endpoints of another arc. 
Let $\mc F_n(q,s)=\{(\sfrac{p}{q},\sfrac{r}{s}),(\sfrac{(s-r)}{s},\sfrac{(q-p)}{q})\}$
 be the set of two Farey pairs associated with $\mc K_n(q,s)$, and  $
{\mc F}_n(\hat q,\hat s)=\{(\sfrac{\hat p}{\hat q}, \sfrac{\hat r}{\hat s}), (\sfrac{(\hat  s-\hat r)}{\hat s}, \sfrac{(\hat q-\hat p)}{\hat q})\}$ the set of two Farey pairs associated with $\mc K_n(\hat q,\hat s)$. If $\mc K_n(q,s)=\mc K_n(\hat q,\hat s)^c$, then the endpoints of the arcs in $\mc K_n(\hat q,\hat s)^c$ have to map to the endpoints of the arcs in $\mc K_n(q,s)$. In terms of the associated Farey pairs: 
\begin{equation}\label{eq:frac1}
\{(\sfrac{\hat p c}{\hat q}-\lfloor\sfrac{\hat p c}{\hat q}\rfloor, \sfrac{\hat rc}{\hat s}-\lfloor\sfrac{\hat rc}{\hat s}\rfloor), (\sfrac{(\hat s-\hat r)c}{\hat s}-\lfloor\sfrac{(\hat s-\hat r)c}{\hat s}\rfloor,\sfrac{(\hat q-\hat p)c}{\hat q}-\lfloor\sfrac{(\hat q-\hat p)c}{\hat q}\rfloor)\}={\mc F}_n(q,s).
\end{equation}
Since arcs in $\mc K_n(q,s)$ span only a fraction of the unit circle, we also need: 
\begin{equation}\label{eq:frac2}
    \lfloor \sfrac{\hat pc}{\hat q}\rfloor= \lfloor\sfrac{\hat rc}{\hat s}\rfloor
\text{ and } 
\lfloor \sfrac{(\hat q-\hat p)c}{\hat q} \rfloor=\lfloor\sfrac{(\hat s-\hat r)c}{\hat s} \rfloor . 
\end{equation}
The endpoints of $\mc K_n(q,s)$ are the $c$-th powers of the endpoints of $\mc K_n(\hat q,\hat s)$ precisely when \eqref{eq:frac1} and \eqref{eq:frac2} hold. Since those will from now on be ongoing assumptions, we gather them in the definition below.  
 \begin{definition}
If $\mc F_n(q,s)$ and $\mc F_n(\hat q, \hat s)$ satisfy equations \eqref{eq:frac1} and \eqref{eq:frac2}, then we define $\mc F_n(\hat q, \hat s) \star c:=\mc F_n(q,s)$, otherwise we say that $\mc F_n(\hat q, \hat s) \star c$ is not defined.
\end{definition}

The example below illustrates that it is possible for \eqref{eq:frac1} to hold while \eqref{eq:frac2} does not.

\begin{example}
Let $\hat q=5$ and $\hat s=6$ and $6 \leq n \leq 10$. Then $\hat p=4$, $\hat r=5$ and $\mc F_n(\hat q, \hat s)=\mc F_n(5,6)=\{(\sfrac{4}{5},\sfrac{5}{6}),(\sfrac{1}{6},\sfrac{1}{5})\}$. 
Taking $c=31$, we verify \eqref{eq:frac1}: $$\{(\sfrac{(4\times 31)}{5} -24,\sfrac{(5\times 31)}{6}-25),(\sfrac{(1\times 31)}{6}-5,\sfrac{(1 \times 31)}{5}-6)\}=\mc F_n(5,6)=\mc F_n(q,s).$$ Clearly,  \eqref{eq:frac2} does not hold and $\mc F_n(5,6)\star 31$ is not defined.  
\end{example}

\begin{lemma}\label{lem:Fpairs}
If $\mc F_n(\hat q, \hat s)\star c=\mc F_n(q,s)$, then either $c$ divides $\hat q$ or $c$ divides $\hat s$. 
\end{lemma}

\begin{proof}
Let us define $c_{\hat q}:=\gcd(c,\hat q)$, $c_{\hat s}:=\gcd(c,\hat s)$ so that $c=c_{\hat q} c_{\hat s} c_0$. If $\mc F_n(\hat q, \hat s)\star c=\mc F_n(q,s)$, we notice that $\{q,s\}=\{\sfrac{\hat q}{c_{\hat q}},\sfrac{\hat s}{c_{\hat s}}\}.$ Now, if $c_{\hat s}=c_{\hat q}=1$ then $\{q,s\}=\{\hat q, \hat s\}$ i.e $\mc F_n(\hat q, \hat s)=\mc F_n(q,s)$ since $q$, $s$ and $n$ uniquely define $\mc F_n(q,s)$. As this implies $c=1$, we can assume $c_{\hat s} c_{\hat q}>1$.

From $\hat q<\hat s \leq n$ and
$q+s=\sfrac{\hat q}{c_{\hat q}}+\sfrac{\hat s}{c_{\hat s}}>n,$
we conclude that $c_{\hat q}\geq 2$ and $c_{\hat s}\geq 2$ cannot occur. 
Assume $c_{\hat s}=1$, $c_{\hat q}\geq 2$, and $\hat q=c_{\hat q}\hat q_0$. (The case $c_{\hat s}\geq 2$ and $c_{\hat q}=1$ can be argued similarly.) By assumption
$$((\sfrac{\hat p c}{\hat q})-a, (\sfrac{\hat rc}{\hat s})-a)=(\sfrac{(\hat p c_0-a \hat q_0)}{\hat q_0}, \sfrac{(\hat rc_{\hat q}c_0-a \hat s)}{\hat s})$$ is a Farey pair for  $a=\lfloor \sfrac{\hat pc}{\hat q}\rfloor= \lfloor\sfrac{\hat rc}{\hat s}\rfloor$. 
Hence :
\begin{align*}
1&=\hat q_0(\hat rc_{\hat q}c_0-a \hat s)-\hat s(\hat p c_0-a \hat q_0) \\
 &=(\hat s\hat p-\hat q \hat r)c_0=c_0.
\end{align*}
We conclude that $c_0=1$, proving that $c$ divides $\hat q$.
\end{proof}

\begin{example}
Let $6\leq n \leq 10$, $\mc F_n(\hat q, \hat s)=\mc F_n(5,6)=\{(\sfrac{4}{5},\sfrac{5}{6}),(\sfrac{1}{6},\sfrac{1}{5})\}$, and $c=3$.
Although $c$ divides $\hat s$,  $\mc F_n(\hat q, \hat s)\star c$ is not defined for all $6\leq n \leq 10$. Indeed, 
$$\{(\sfrac{(4\times 3)}{5} -2,\sfrac{(5\times 3)}{6}-2),(\sfrac{(1\times 3)}{6}-0,\sfrac{(1 \times 3)}{5}-0)\}=\{(\sfrac{2}{5},\sfrac{1}{2}),(\sfrac{1}{2}, \sfrac{3}{5})\},$$
and $\{(\sfrac{2}{5},\sfrac{1}{2}),(\sfrac{1}{2}, \sfrac{3}{5})\}=\mc F_n(2,5)$ for $n\in \{5,6\}$ but not for $n \in \{7,8,9,10\}$. 
\end{example}

\begin{corollary}\label{corr:endpoints}
Let $\mc F_n(\hat q, \hat s)\star c=\mc F_n(q,s)$, ${\hat \zeta}=(\frac{ \hat p}{\hat q},\frac{\hat r}{\hat s}) \in \mc F_n(\hat q,\hat s)$ and $a:=\lfloor \sfrac{\hat pc}{\hat q}\rfloor= \lfloor\sfrac{\hat rc}{\hat s}\rfloor$.
\begin{enumerate}
    \item If $c$ divides $\hat q$, then:
    \begin{itemize}
    \item ${\zeta}=(\frac{ p}{ q},\frac{r}{ s}) \in \mc F_n( q, s)$ where $\hat s=s$, $\hat q=cq$, $\hat p=p+a q$ and $c\hat r=r+a s$,
        \item $\hat \theta \in (\frac{2 \pi\hat p}{\hat q},\frac{2\pi \hat r}{\hat s})$ if and only if $\hat \theta=\frac{1}{c}(\theta+2\pi a)$ for $\theta \in (\frac{2\pi p}{ q},\frac{2\pi r}{ s})$.
    \end{itemize}
    \item If $c$ divides $\hat s$, then:
    \begin{itemize}
    \item ${\zeta}=(\frac{ s-r}{ s},\frac{q-p}{ q}) \in \mc F_n( q, s)$ where $\hat q=s$, $\hat s=cq$, $\hat pc=s-r+a s$ and $\hat r=q-p+a q$, 
    \item $\hat \theta \in (\frac{2 \pi\hat p}{\hat q},\frac{2\pi \hat r}{\hat s})$ if and only if $\hat \theta=\frac{1}{c}(\theta'+2\pi a)$ for $\theta' \in (\frac{2\pi (s-r)}{ s},\frac{2\pi(q-p)}{q})$.
    \end{itemize}
\end{enumerate}
\end{corollary}

\begin{proof}
\begin{enumerate}
    \item Assuming ${\hat\zeta}=(\frac{ \hat p}{ \hat q},\frac{\hat r}{ \hat s}) \in \mc F_n(\hat q, \hat s),$ $c|\hat q$ and $\hat q=c\hat q_0$, we have:
    $$\left (\frac{\hat p c}{\hat q}-a, \frac{\hat rc}{\hat s}-a\right )=\left (\frac{\hat p-a \hat q_0}{\hat q_0}, \frac{\hat r c-a \hat s}{\hat s}\right ) \in \mc F_n(q,s)$$
    by \eqref{eq:frac1}. Since $\hat q_0<\hat q<\hat s$, this implies: 
$\hat q_0=q$, $\hat s=s$, $\hat p-a \hat q_0=p$, and $c\hat r- a \hat s=r$, as desired. Now:
$$ \left (\frac{2 \pi\hat p c}{\hat q},\frac{2\pi \hat r c}{\hat s}\right )= \left (\frac{2 \pi(p+a q) }{q},\frac{2\pi(r+a s)}{s}\right)=\left (\frac{2 \pi p }{q}+2 \pi a,\frac{2\pi r}{s}+ 2 \pi a\right).$$
The second part of the claim follows. 
\item Let $c|\hat s$, then $\hat s=c\hat s_0$ which gives:

$$\left(\frac{\hat p c}{\hat q}-a, \frac{\hat rc}{\hat s}-a \right)=\left (\frac{\hat p c- a \hat q}{\hat q}, \frac{\hat r-a\hat s_0}{\hat s_0} \right)\in \mc F_n(q,s).$$ 

Since $\hat q< \hat s$, we have  either $\hat s_0<\hat q$ or $\hat q<\hat s_0<\hat s$. The latter implies $n<2\hat s_0$ i.e $n<2 \hat s/c$ which is not possible as $c\geq 2$. Thus, taking $\hat q>\hat s_0$ we must have $\hat q=s$, $\hat s=cq$, $\hat p c=s-r+as$ and $\hat r=q-p+aq.$

Further,  
\begin{align*}
     \left (\frac{2 \pi\hat p c}{\hat q},\frac{2\pi \hat r c}{\hat s}\right )   =\left( \frac{2\pi (s-r)}{s}+2\pi a, \frac{2\pi(q-p)}{q}+2\pi a\right) .\\
   \end{align*}
 
Thus, $\hat \theta c=\theta'+2\pi a $, where $\theta' \in ( \frac{2\pi (s-r)}{s}, \frac{2\pi(q-p)}{q})$.
\end{enumerate}
\end{proof}

\begin{example}
Let $\mc F_n(\hat q, \hat s)=\mc F_6(5,6)=\{(\sfrac{4}{5},\sfrac{5}{6}),(\sfrac{1}{6},\sfrac{1}{5})\}$, $c=3$, and $a:=\lfloor \frac{\hat pc}{\hat q}\rfloor=\lfloor \frac{\hat rc}{s}\rfloor=2$. Then $ F_6(5,6)\star 3=\mc F_6(2,5)$. In particular, 
$$(e^{\frac{4\cdot 2 \pi i }{5}})^3=e^{\frac{2\cdot 2 \pi i }{5}} \text{, } (e^{\frac{5\cdot 2 \pi i }{6}})^3=e^{\frac{ 2 \pi i }{2}},$$
and  $\hat \theta \in (\frac{4 \cdot 2 \pi}{5},\frac{5\cdot 2 \pi }{6})$ if and only if $\theta'=3\hat \theta-2 \cdot 2\pi\in (\frac{2 \cdot 2\pi}{5},\frac{2\pi}{2})$.
 \end{example}

Now that we understand, when the endpoints of a Karpelevi\v c arc are powers of endpoints of another Karpelevi\v c arc, we want to know when the whole arc is mapped to another arc. To this end, we compute the derivative of the modulus $\rho_n$ with respect to the argument $\theta$ of a Karpelevi\v c arc at the endpoints. 

\begin{lemma} \label{lem:deriv} 
Suppose that  $(\frac{p}{q},\frac{r}{s})$, $q < s$, is a Farey pair and for  $\theta \in [\sfrac{2 \pi p}{q},\sfrac{2 \pi r}{s}]$ let the point on the boundary of $\Theta_n$ with the argument $\theta$ be given by $\rho_n(\theta) e^{\ii \theta}$. 
Then 
	\begin{equation}\label{eq:derivq}
	 \frac{\partial \rho_n(\theta)}{\partial \theta}\Big |_{\theta =\frac{2 \pi p}{q} }\sin \left(\frac{2 \pi }{qd} \right)=\cos \left(\frac{2 \pi }{qd}\right)-1	\end{equation}
	 
	 and 
	 \begin{equation}\label{eq:derivs}
	 \frac{\partial \rho_n(\theta)}{\partial \theta}\Big |_{\theta =\frac{2 \pi r}{s} }\sin \left(\frac{2 \pi }{s} \right)=1-\cos \left(\frac{2 \pi }{s}\right).	\end{equation}
\end{lemma}

\begin{proof} 
Implicit differentiation of equation \eqref{eq:implicit} gives us:
\begin{align*}
\mu ^{s_1} q \cos(q \theta)+s_1\mu^{s_1-1} \frac{\partial \mu(\theta)}{\partial \theta}\Big |_{\theta}\sin (q\theta)-\\
\Big[ \mu^{qd_1} \frac{s_1}{d_1}\cos\left(\frac{s_1}{d_1}\theta  -\frac{2 \pi r}{d} \right)+qd_1\mu^{qd_1-1}\frac{\partial \mu(\theta)}{\partial \theta}\Big |_{\theta}\sin\left(\frac{s_1}{d_1}\theta  -\frac{2 \pi r}{d} \right)  \Big]-\\
\Big[(q-\frac{s_1}{d_1})\cos\left((q-\frac{s_1}{d_1})\theta+\frac{2 \pi r}{d} \right)  \Big]=0.
\end{align*}

Taking $\mu=1$ we have:
\begin{align*}
 q\cos(q\theta)+s_1\frac{\partial \mu(\theta)}{\partial \theta}\Big |_{\theta}\sin(q\theta)-\\
  \Big[\frac{s_1}{d_1}\cos\Big(\frac{s_1}{d_1}\theta-\frac{2\pi r}{d}\Big)+qd_1 \frac{\partial \mu(\theta)}{\partial \theta} \sin\Big(\frac{s_1}{d_1}\theta-\frac{2\pi r}{d}\Big) \Big]-\\
 \Big[(q-\frac{s_1}{d_1})\cos\left((q-\frac{s_1}{d_1})\theta+\frac{2 \pi r}{d} \right)  \Big]=0.
\end{align*}

Inserting $\theta=\frac{2 \pi p}{q}$ and $\theta=\frac{2 \pi r}{s}$ in above we get:
\begin{align*}
    \frac{\partial \mu(\theta)}{\partial \theta}\Big |_{\theta=2\pi p/q}\sin\left( \frac{2\pi}{qd}\right)=\frac{1}{d_1}\left(\cos \left( \frac{2\pi}{qd}\right)-1\right)
\end{align*}
and
\begin{align*}
  \frac{\partial \mu(\theta)}{\partial \theta}\Big |_{\theta=2\pi r/s}\sin\left( \frac{2\pi}{s}\right)=\frac{1}{d_1}\left(1-\cos \left( \frac{2\pi}{s}\right)\right). 
\end{align*}
 
Noting that, $\rho_n(\theta)=\mu(\theta)^{d_1}$ which implies $\frac{\partial \rho_n(\theta)}{\partial \theta}=d_1 \mu(\theta)^{d_1-1}\frac{\partial \mu}{\partial \theta}$,  equations \eqref{eq:derivq} and \eqref{eq:derivs} follow. 
\end{proof}

\begin{corollary}\label{cor:derivative}
Let $\mc K_n(q,s)=\mc K_n(\hat q,\hat s)^c$,  $d:=\lfloor\sfrac{n}{q}\rfloor$ and  $\hat d:=\lfloor\sfrac{n}{\hat q}\rfloor$. Then:
\begin{enumerate}
    \item If $c$ divides $\hat q$, then $qd=\hat q \hat d$ and $s=\hat s$.
    \item If $c$ divides $\hat s$, then $qd=\hat s$ and $s=\hat q \hat d$.
\end{enumerate}
\end{corollary}

\begin{proof}
Let $F(x):=\frac{1-\cos(x)}{\sin(x)}$. Note that $F(x)$ is injective on $[0,2\pi]$. 

If $\mc K_n(q,s)=\mc K_n(\hat q,\hat s)^c$, then for all $\theta \in (\frac{2 \pi p}{q}, \frac{2 
\pi r}{s})$ we have  $\rho_n(\theta)=(\rho_n(\hat \theta))^c$ for $\hat \theta \in \arg(\hat q, \hat s)$, where by Corollary \ref{corr:endpoints}:
\begin{equation}\label{eq:thetarelq}
    \hat \theta=\frac{1}{c}\left (\theta+2 \pi a\right ) \text{ if } c \text{ divides } \hat q, 
    \end{equation}
    \begin{equation}\label{eq:thetarels}
    \hat \theta=\frac{1}{c}\left (2\pi - \theta+2 \pi a\right) \text{ if } c \text{ divides } \hat s.
\end{equation}
In particular:
\begin{equation}\label{eq:rhoder1}
  \frac{\partial (\rho_n(\hat \theta))^c}{\partial \theta}\Big |_{\theta =\frac{2 \pi p}{q}}
        =\frac{\partial \rho_n(\theta)}{\partial \theta}\Big |_{\theta =\frac{2 \pi p}{q}},
\end{equation}
\begin{equation}\label{eq:rhoder2}
 \frac{\partial (\rho_n(\hat \theta))^c}{\partial \theta}\Big |_{\theta =\frac{2 \pi r}{s}}
        =\frac{\partial \rho_n(\theta)}{\partial \theta}\Big |_{\theta =\frac{2 \pi r}{s}},
\end{equation}
and
  \begin{align*}
    \frac{\partial  (\rho_n(\hat \theta))^c}{\partial  \theta}
        &=c \rho_n(\hat \theta)^{c-1}\frac{ \cdot \partial \rho_n(\hat \theta)}{\partial \hat \theta}       \frac{\partial \hat\theta}{\partial \theta}.
        \end{align*}
 \begin{enumerate}
    \item If $c$ divides $\hat q$, then :
     $$\frac{\partial \hat\theta}{\partial \theta}=\frac{1}{c},$$
     by equation \eqref{eq:thetarelq}.
     Now, by Lemma by \ref{lem:deriv}:
    \begin{align*}
     \frac{\partial \rho_n(\theta)}{\partial  \theta}\Big |_{\theta =\frac{2 \pi p}{q}}&= \frac{\cos \left(\frac{2 \pi }{qd}\right)-1}{\sin \left(\frac{2 \pi }{qd}\right)}=-F\left (\frac{2 \pi }{qd}\right ),\\
    \frac{\partial  (\rho_n(\hat \theta))^c}{\partial  \theta}\Big |_{\theta =\frac{2 \pi p}{q} }&=\frac{\partial \rho_n(\hat \theta)}{\partial \hat \theta}\Big |_{\hat \theta =\frac{2 \pi \hat p}{\hat q} }=\frac{\cos \left(\frac{2 \pi }{\hat q \hat d}\right)-1}{\sin \left(\frac{2 \pi }{\hat q \hat d}\right)}=-F\left (\frac{2 \pi }{\hat q\hat d}\right ), 
\end{align*}
since $\rho_n(\hat \theta)=1$ at the endpoints, and $\hat \theta=\sfrac{2\pi \hat p}{\hat q}$ when $\theta=\sfrac{2\pi p}{q}$. Hence, $F\left (\frac{2 \pi }{qd}\right )=F\left (\frac{2 \pi }{\hat q\hat d}\right )$ by equation \eqref{eq:rhoder1} and by injectivity of $F$ we have $\hat q \hat d=qd$.

Similarly, 
$
\frac{\partial \rho_n(\theta)}{\partial  \theta}\Big |_{\theta =\frac{2 \pi r}{s}}
=F\left (\frac{2 \pi }{s}\right )
$ and $\frac{\partial (\rho_n(\hat \theta))^c}{\partial \theta}\Big |_{\theta =\frac{2 \pi r}{s}} =F\left (\frac{2 \pi }{\hat s}\right ),
$
which implies $\hat s=s$.

\item  If $c$ divides $\hat s$, then :
     $$\frac{\partial \hat\theta}{\partial \theta}=-\frac{1}{c},$$
     by equation \eqref{eq:thetarels}.
     Now, by Lemma by \ref{lem:deriv}:
    \begin{align*}
     \frac{\partial \rho_n(\theta)}{\partial  \theta}\Big |_{\theta =\frac{2 \pi p}{q}}&= \frac{\cos \left(\frac{2 \pi }{qd}\right)-1}{\sin \left(\frac{2 \pi }{qd}\right)}=-F\left (\frac{2 \pi }{qd}\right ),\\
    \frac{\partial  (\rho_n(\hat \theta))^c}{\partial  \theta}\Big |_{\theta =\frac{2 \pi p}{q} }&=-\frac{\partial \rho_n(\hat \theta)}{\partial \hat \theta}\Big |_{\hat \theta =\frac{2 \pi \hat r}{\hat s} }=\frac{\cos \left(\frac{2 \pi }{\hat s}\right)-1}{\sin \left(\frac{2 \pi }{\hat s}\right)}=-F\left (\frac{2 \pi }{\hat s}\right ), 
\end{align*}
since $\rho_n(\hat \theta)=1$ at the endpoints, and $\hat \theta=\sfrac{2\pi \hat r}{\hat s}$ when $\theta=\sfrac{2\pi p}{q}$. Hence, $F\left (\frac{2 \pi }{qd}\right )=F\left (\frac{2 \pi }{\hat s}\right )$ by equation \eqref{eq:rhoder1} and by injectivity of $F$ we have $\hat s=qd$. Similarly, identities
$
\frac{\partial \rho_n(\theta)}{\partial  \theta}\Big |_{\theta =\frac{2 \pi r}{s}}
=F\left (\frac{2 \pi }{s}\right )
$ and
$\frac{\partial (\rho_n(\hat \theta))^c}{\partial \theta}\Big |_{\theta =\frac{2 \pi r}{s}}=F\left (\frac{2 \pi }{\hat q \hat d}\right )
$
imply $\hat q \hat d=s$.
\end{enumerate}
\end{proof}

The following example shows that $\mc F_n(\hat q, \hat s)\star c=\mc F_n(q,s)$ does not imply $\mc K_n(\hat q, \hat s)^c=\mc K_n(q,s)$.
\begin{example}
Let $n=27$, $\mc F_n(\hat q, \hat s)=\mc F_{27}(4,27)$, $\mc F_n(q,s)=\mc F_{27}(2,27)$ and $c=2$. Then:
$$\mc F_{27}(4,27)\star 2=\mc F_{27}(2,27).$$
We have $q=2$, $\hat q=4$, $d=13$ and $\hat d=6$. Clearly $c$ divides $\hat q$ and $qd =26\ne \hat q \hat d=24 $. Therefore $\mc K_n(\hat q, \hat s)^c=\mc K_{27}(4,27)^2 \ne \mc K_{27}(2,27)=\mc K_n(q,s)$ by Corollary \ref{cor:derivative}.
\end{example}

\begin{theorem}\label{thm:powers}
Let $\mc K_n(q,s) \in \mc K_n$ and $d:=\lfloor\sfrac{n}{q}\rfloor$. Then $\mc K_n(q,s)=\mc K_n(\hat q, \hat s)^{c}$ if, and only if, one of the following situations occurs:
\begin{enumerate}
    \item\label{thm:powers-i1} $\mc K_n(q,s)=\mc K_n(c q, s)^c$, for any $c$ that divides $d$, satisfies $\gcd(c,s)=1$ and $cq<s$.
    \item\label{thm:powers-i2} $\mc K_n(q,s)=\mc K_n(s,qd)^d$ if $\delta:=\gcd(d,s)=1$.
    \end{enumerate}
    \end{theorem}
    
\begin{proof}
Let $\mc K_n(q,s) \in \mc K_n$ and $\mc K_n(q,s)=\mc K_n(\hat q, \hat s)^{c}$. Then $\mc F_n(\hat q, \hat s)\star c=\mc F_n(q,s)$, and either $c$ divides $\hat q$ or $c$ divides $\hat s$ by Lemma \ref{lem:Fpairs}. We consider each case separately. 

Assume first that $c$ divides $\hat q$ (and $\gcd(c,\hat s)=1$). Hence, $\hat q=q c$ and $\hat s=s$ by Corollary \ref{corr:endpoints}. Clearly, we also have   $\gcd(c,s)=1$. By Corollary \ref{cor:derivative} we now have  $qd=qc \hat d$ which implies $\hat{d}c=d$. Since $\hat q<\hat s$, $qc<s$ has to hold. With this, we have shown that if $\mc K_n(q,s)=\mc K_n(\hat q, \hat s)^{c}$ and $c$ divides $\hat q$, then the conditions listed in item \ref{thm:powers-i1} have to hold. 

Now, assuming $c$ divides $d$, $\gcd(c,s)=1$, and $cq<s$, we want to prove $\mc K_n(q,s)=\mc K_n(c q, s)^c$. Let $\theta \in (\sfrac{2 \pi p}{q},\sfrac{2 \pi r}{s})$, where $(\sfrac{p}{q},\sfrac{r}{s}) \in \mc F_n(q,s)$. Theorem \ref{thm:gcd200} tells us that  $\rho_n(\theta)e^{i\theta}=\mu^{d_1} e^{i \theta} \in \mc K_n(q,s)$ if and only if $\mu$ is the unique solution to \eqref{eq:implicit}. 

To study $\hat \lambda \in \mc K_n(\hat q, \hat s)^{c}$ with the argument $\theta$, we first use Corollary \ref{corr:endpoints} to  identify parameters associated with Farey pairs in $\mc K_n(\hat q, \hat s)$:
\begin{equation}\label{eq:hat_parameters1}
\hat s=s, \, \hat q=cq, \, \hat p=p+a q, \, \text{ and } c\hat r=r+a s,
\end{equation}
and note that $\hat \theta:=\frac{1}{c} (\theta+2 \pi a) \in (\sfrac{2 \pi\hat p}{\hat q},\sfrac{2 \pi\hat r}{\hat s})\subset \arg(\hat q, \hat s)$. Now, $\gcd(c,s)=1$ implies $\gcd(c \hat d,s)=\gcd(\hat d,s)$ which gives us
$\hat \delta=\gcd(\hat d,\hat s)=\gcd(c \hat d,s)=\gcd(d,s)=\delta.$ 
Hence: 
\begin{equation}\label{eq:hat_parameters2}
    \hat \delta=\delta, \, \hat s_1=s_1,\, \hat d_1c=d_1, 
\end{equation}
where $\hat s_1:=s_1(\hat q, \hat s)$ and $\hat d_1:=d_1(\hat q,\hat s)$. 

By Theorem  \ref{thm:gcd200}, $\hat \mu^{\hat d_1} e^{i \hat \theta}\in \mc K_n(c q,s)$ if and only if $\hat \mu$ satisfies equation $\eqref{eq:implicit}$ for $\hat \theta$ and the parameters associated with $(\sfrac{\hat p}{\hat q},\sfrac{\hat r}{\hat s})$:
\begin{align}\label{eq:hat_implicit}
\hat \mu^{\hat s_1}\sin(\hat q \hat \theta)-\hat \mu^{\hat q \hat d_1} \sin\Big(\frac{\hat s_1}{\hat d_1}\hat \theta-\frac{2\pi \hat r}{\hat d}\Big)-\sin\Big((\hat q-\frac{\hat s_1}{\hat d_1})\hat \theta+\frac{2\pi \hat r}{\hat d}\Big)=0.
\end{align}
Using \eqref{eq:hat_parameters1} and  \eqref{eq:hat_parameters2} we get:
\begin{align*}
\hat \mu^{\hat s_1} \sin(\hat q\hat \theta)&=\hat \mu^{s_1}\sin\left(q(\theta+2 
\pi a) \right)=\hat \mu^{s_1} \sin(q \theta), \\
\mu^{\hat q\hat d_1} \sin \left(\frac{\hat s_1}{\hat d_1}\hat \theta  -\frac{2 \pi \hat r}{\hat d} \right)&={\hat\mu}^{ qd_1} \sin \left(\frac{s_1}{d_1}(\theta+2 
\pi a)  -\frac{2 \pi (r+a s)}{d} \right) \\
&={\hat\mu}^{ qd_1} \sin \left(\frac{s_1}{d_1}\theta-\frac{2 \pi r}{d} \right),\\
\sin\left((\hat q-\frac{\hat s_1}{\hat d_1})\hat \theta+\frac{2 \pi \hat r}{\hat d} \right)&=\sin\left((c q-\frac{c s_1}{ d_1})\frac{(\theta +2 \pi a)}{c}+\frac{2 \pi (r+a s)}{d} \right)\\
&=\sin\left((q-\frac{s_1}{d_1})\theta+\frac{2 \pi r}{d} \right),
\end{align*}
which proves that the coefficients of $
\hat \mu$ in \eqref{eq:hat_implicit} and the coefficients of $\mu$ in \eqref{eq:implicit} agree. Since
\eqref{eq:implicit} defines $\mu$ uniquely, we conclude $\hat \mu=\mu$. Finally, 
$$\hat \lambda^c=(\hat \mu^{\hat d_1}e^{i \hat \theta})^c=\mu^{d_1} e^{i \theta}=\rho_n(\theta)e^{i \theta}$$
completes the proof that $\mc K_n(cq,s)^c=\mc K_n(q,s)$.

Next, assume that $c$ divides $\hat s$ (and $\gcd(c, \hat q)=1$). Hence, $\hat s= c q$ and $\hat q=s$ by Corollary \ref{corr:endpoints}. Now, by Corollary \ref{cor:derivative} we have $qd=cq$ and $s=s \hat d$ which gives $c=d$ and $\hat d=1$. Therefore, $\delta=\gcd(d,s)=\gcd(c,\hat q)=1$, as required. This shows that if $\mc K_n(q,s)=\mc K_n(\hat q, \hat s)^{c}$ and $c$ divides $\hat s$, then the condition in item \ref{thm:powers-i2} has to hold.

Now, assuming $\gcd(d,s)=1$, we want to prove $\mc K_n(q,s)=\mc K_n(s,qd)^d$. Since, $c$ divides $\hat s$, let $\theta' \in  (\sfrac{2 \pi (s-r)}{s},\sfrac{2 \pi (q-p)}{q})$, where $\zeta=(\sfrac{(s-r)}{s},\sfrac{(q-p)}{q}) \in \mc F_n(q,s)$. By Remark  \ref{remark:gcd200}, $\rho_n e^{i\theta'}=\mu^{d_1} e^{i \theta'}\in \mc K_n(q,s)$ if and only if $\mu$ satisfies $\eqref{eq:implicit}$ for parameters associated with $\zeta$ and $\theta=2 \pi-\theta'$. That is:

$$
\mu^{s_1} \sin(q(2\pi -\theta'))- \mu^{qd_1} \sin \left(\frac{s_1}{d_1}(2\pi-\theta')  -\frac{2 \pi r}{d} \right)-\sin\left((q-\frac{s_1}{d_1})(2
\pi -\theta')+\frac{2 \pi r}{d} \right)=0,$$
or equivalently:

\begin{align}\label{conjugateimplicit}
-\mu^{s_1} \sin(q\theta')+ \mu^{qd_1} \sin \left(\frac{s_1}{d_1}\theta'  +\frac{2 \pi (r-s)}{d} \right)+\sin\left((q-\frac{s_1}{d_1})\theta'+\frac{2 \pi (s-r)}{d} \right)=0. 
\end{align}

To understand $\hat \lambda \in \mc K_n(\hat q, \hat s)^{c}$ with the argument $\theta'$, we first use Corollary \ref{corr:endpoints} to  identify parameters associated with Farey pairs in $\mc K_n(\hat q, \hat s)$:
\begin{equation}\label{eq:hat_parameters3}
\hat q=s, \, \hat s=qc, \, \hat p c=s-r+as, \, \text{ and }\hat r=q-p+aq,
\end{equation}
and $\hat \theta=\frac{1}{c}(\theta'+2\pi a)\in (\sfrac{2 \pi\hat p}{\hat q},\sfrac{2 \pi\hat r}{\hat s})\subset \arg(\hat q, \hat s) $. In addition:  $c=d$, $\hat d=1$, and  $\hat d_1=1.$

Theorem  \ref{thm:gcd200} tells us that $\hat \lambda= \hat \mu^{\hat d_1} e^{i \hat \theta} \in \mc K_n(s,qd)$ if and only of $\hat \mu$ satisfies \eqref{eq:hat_implicit} for parameters associated with $(\sfrac{\hat p}{\hat q},\sfrac{\hat r}{\hat s})$ and $\hat \theta$. Using the parameters defined above, we get: 
\begin{align*}
    \hat \mu^{\hat s_1} \sin(\hat q\hat \theta)&=\hat \mu^{qd}\sin\left(\frac{s}{d}\theta'+2\pi\frac{as}{d} \right),\\
\hat\mu^{\hat q\hat d_1} \sin \left(\frac{\hat s_1}{\hat d_1}\hat \theta  -\frac{2 \pi \hat r}{\hat d} \right)&={\hat\mu}^{s} \sin \left(q\theta' \right),\\
\sin\left((\hat q-\frac{\hat s_1}{\hat d_1})\hat \theta+\frac{2 \pi \hat r}{\hat d} \right)&=\sin\left((\frac{s}{d}-q)(\theta'+2 \pi a) \right)\\
&=\sin\left((\frac{s}{d}-q)\theta' + \frac{2 \pi as}{d} \right).
\end{align*}
Since $as=\hat pd-s+r$ implies that the coefficients of $\hat \mu$ in \eqref{eq:hat_implicit} and the coefficients of $\mu$ in \eqref{conjugateimplicit} are the same, we can conclude that $\hat \mu=\mu$. Hence, $\hat \lambda^d=(\hat \mu^{\hat d_1}e^{i \hat \theta})^d=\mu^{d_1} e^{i\theta'}=\rho_n(\theta)e^{i\theta'}$ i.e $\mc K_n(q,s)=\mc K_n(s,qd)^d$, as required.
\end{proof}

\begin{example}
Theorem \ref{thm:powers} allows us to identify all Karpelevi\v c arcs that are powers of another arc. In particular, for $n=8$ we list all such cases below:
\begin{enumerate}
    \item $\mc K_8(4,7)=\mc K_8(7,8)^2$
    \item $\mc K_8(2,7)=\mc K_8(7,8)^4=K_8(4,7)^2$
    \item $\mc K_8(3,7)=\mc K_8(7,6)^2$
    \item $\mc K_8(4,5)=\mc K_8(5,8)^2.$
\end{enumerate}

As an example let us consider $\mc K_8(4,7)$, with $q=4$, $s=7$, $d=2$ and $\delta=1$. Hence, $\mc K_8(4,7)=\mc K_8(7,8)^2$, by item \ref{thm:powers-i1} in the theorem. 
Considering $\mc K_8(2,7)$ with $q=2$, $s=7$, $d=4$, $\delta=1$, we have $\mc K_8(2,7)=\mc K_8(7,8)^4$ by item \ref{thm:powers-i2}. Also, $c=2$ divides $d=4$ which implies $\mc K_8(2,7)=\mc K_8(4,7)^2$ by item \ref{thm:powers-i1}. 
\end{example}

Theorem \ref{thm:powers} answers the question, given $\mc K_{n}(q,s)$, what are all possible arcs $\mc K_{n}(\hat q,\hat s)$ whose power is $\mc K_{n}(q,s)$. Since the theorem gives the complete characterisation, it also answers the question, given $\mc K_{n}(\hat q,\hat s)$, what are all possible arcs $\mc K_{n}(q,s)$ that are powers of $\mc K_{n}(\hat q,\hat s)$. This point of view is given in the result below. 
\begin{corollary}\label{cor:powers1}
Let $\mc K_n(\hat q,\hat s) \in \mc K_n$ and $\hat d:=\lfloor\sfrac{n}{q}\rfloor$. Then $\mc K_n(\hat q, \hat s)^{c}=\mc K_n(q,s) \in \mc K_n$ if and only if one of the following situations occurs:
\begin{enumerate}
    \item\label{thm:powers-i1} $\mc K_n(\hat q, \hat s)^c=\mc K_n(\sfrac{\hat q}{c},\hat s)$ for all $c$ that divide $\hat q$ and satisfy $|\hat s-\hat q \hat d|<\sfrac{\hat q}{c}$.
    \item\label{thm:powers-i2} $\mc K_n(\hat q,\hat s)^c=\mc K_n(\sfrac{\hat s}{c},\hat q)$, for all $c$ that divide $\hat s$ and satisfy $n-\hat q<\sfrac{\hat s}{c}$.
    \end{enumerate}
\end{corollary}

\section{Stochastic matrices and  associated digraphs, notation and background}\label{sec:background2}

In the second part of the paper, we consider the question, of when a stochastic matrix, that realises an eigenvalue on the border of the Karpelevi\v c region, can be written as a power of another stochastic matrix. This section is dedicated to the necessary background and notation.

\subsection{Notation}

When taking powers of matrices and the associated digraphs, we will repeatedly encounter modular arithmetic. Since we want to use the standard numbering of the rows and columns of $n \times n$ matrices from $1$ to $n$, it is convenient to define: 
$\langle k \rangle_n:=1+((k-1) \mod n)$. In this notation $\langle k \rangle_n \in \{1,\ldots,n\}$. 

Given $n \in \N$ we define the following vectors:
\begin{align*}
    {\bf a}(n)&:=\npmatrix{1 & 2 & \ldots  & n},\\
    {\bf a_0}(n)&:=\npmatrix{0 & 1 & \ldots  & n-1},\\
    {\bf e}&:=\npmatrix{1 & 1 & \ldots  & 1},
\end{align*}
where the size of ${\bf e}$ will be clear from the context. 
 In addition, we will depend on standard operations to build new vectors.  For example, $i\cdot {\bf a}(n)=\npmatrix{i & 2 i & \ldots  & n i}$, $i\cdot {\bf e}+ {\bf a}(n)=\npmatrix{i+1 &  i+2 & \ldots  & i+n}$, etc.  Furthermore, for ${\bf v}=\npmatrix{v_1 & \ldots v_k} \in ~ \N_0^k$ we denote: 
   $$\T ({\bf v}):=\{\npmatrix{v_i & v_{\langle i+1 \rangle_k} & \ldots & v_{\langle i+k-1\rangle_k}}; i=1,\ldots,k \}$$
   to be the set of vectors obtained from ${\bf v}$ by cyclic permutations of its elements.

A digraph $G$ is defined by its vertex set $V(G)=\{1,\ldots,n\}$ and edge set $E(G)\subseteq V(G)\times V(G)$.  For ${\bf v}=\npmatrix{v_1 & \ldots v_k} \in \N^k$ we denote by $C({\bf v})$ the $k$-cycle with $V(C({\bf v}))=\{{v}_i; i=1,\ldots,k\}$ and $E(C({\bf v}))=\{({v}_i,{v}_{\langle 1+i \rangle_k}); i=1,\ldots, k\}$. Clearly, $C({\bf v})=C({\bf u})$ if and only if $\T({\bf v})=\T({\bf u})$. The \emph{weight of a cycle} is defined to be the product of the weights on the edges of that cycle. Furthermore, $P({\bf v})$ will denote the path with $V(P({\bf v}))=\{{ v}_i; i=1,\ldots,k\}$ and $E(P({\bf v}))=\{({ v}_i,{v}_{i+1}); i=1,\ldots, k-1\}.$
Let $\mc E \subset V(G)\times V(G)$, then $G + \mc E$ is defined to be the digraph with $V(G + \mc E)=V(G)$ and $E(G + \mc E)=E(G)+\mc E$. With $G^{(b)}$ we denote the strong power of $G$, i.e. the digraph on vertex set $V(G)$, where $(v_1,v_2) \in E(G^{(b)})$ if and only if $v_1$ and $v_2$ are at distance $b$ in $G$. Given a nonnegative $n\times n$ matrix $A$, we define $\Gamma(A)$ to be \emph{the digraph associated with $A$} defined by $V(\Gamma(A))=\{1,2,\ldots,n\}$ and $E(\Gamma(A))=\{(i,j); (A)_{i,j} \neq 0\}$.

For integers $q,s,n \in \N$ satisfying $q<s$, $\gcd(q,s)=1$ and $q+s>n$, we denote by  $\mathcal M_n(q,s)$ the set of all $n \times n$ stochastic matrices with an eigenvalue from $\mc K_n(q,s)$. The set of sparsest matrices in $\mathcal M_n(q,s)$ is denoted by $\mathcal M_n^0(q,s)$. More precisely, $A \in \mathcal M_n^0(q,s)$ if and only if $A=(a_{ij}) \in \mathcal M_n(q,s)$ and there does not exist $A'=(a'_{ij})\in \mathcal M_n(q,s)$ satisfying $\{(i,j); a'_{ij} \neq 0\}\subset \{(i,j); a_{ij} \neq 0\}$. (It turns out that, for fixed $n$, $q$ and $s$, all matrices in $\mathcal M^0_n(q,s)$ have the same number of nonzero elements, hence $\mathcal M^0_n(q,s)$ can also be defined as the set of matrices in 
$\mathcal M_n(q,s)$ with the least number of non-zero entries.)

Let $\mc S$ be a set of $n \times n$ matrices and $c \in \N$, then $\mc S^c:=\{A^c; A\in \mc S\}$. For $n \times m$ matrix $M$, let $\vecm(M)$ denote the $m\cdot n$ vector obtained by stacking the columns of the matrix $M$.

\subsection{Sparsest realisations of Ito polynomials}

In  \cite{kirklandsmigoc} the sparsest realizing matrices for Ito polynomials of degree $n$ were characterised. Given a reduced Ito polynomial $f_{\alpha}(t)$ of degree $n$ associated with $\mc K_n(q,s)$, the sparsest realisation of $f_{\alpha}(t)$ is (up to permutation similarity) uniquely defined by the associated digraph. Moreover, associated digraphs are precisely digraphs on $n$ vertices that contain one $s$-cycle and $d$ disjoint $q$-cycles.

Let $q,s,\hat q, \hat s, n$ and $c$ be such that $\mc K_n(q,s)=\mc K_n(\hat q,\hat s)^c$. Then:
\begin{align*}
\mc M_n(\hat q,\hat s)^c&\subset \mc M_n(q,s),\\
\mc M_n^0(\hat q,\hat s)^c&\subset \mc M_n^0(q,s),
\end{align*}
but $\mc M_n(\hat q,\hat s)^c\neq \mc M_n(q,s)$ and  $\mc M_n^0(\hat q,\hat s)^c\neq \mc M_n^0(q,s)$. In this section, we aim to characterise matrices in $\mc M_n^0(\hat q,\hat s)^c$ in terms of the associated digraphs. For this, we first recall 
the characterisation of $\mc M_n^0(q,s)$ given in \cite{kirklandsmigoc} that depends on separating the Karpelevi\v c arcs into four categories: Type 0, I, II, and III. This separation was introduced by Johnson and Paparella \cite{johnsonpaperella}. 

 While it is possible for the degree of the reduced Ito polynomial associated with a Farey pair $(\sfrac{p}{q},\sfrac{r}{s})$ of order $n$ to be less than $n$, we will consider only the situations when $n$ and the degree of the reduced Ito polynomial agree. 
 For example, the reduced Ito polynomials associated with $\mc K_{n}(3,14)$ for $n=15$ and $n=16$ are both equal to $f_{\alpha}(t)=(t^3-(1-\alpha))^5-\alpha^5 t,$ $\alpha \in[0,1]$. But since for $n=16$, this polynomial has a degree less than $n$, our investigation will not cover this case.

 Below we list the different types of reduced Ito polynomials of order $n$,  using the notation introduced in \cite{kirklandsmigoc}:
\begin{itemize}
\item Type 0: If $n=s, d=n, q=1$, then  $f_{\alpha}(t)=(t+\alpha-1)^n-\alpha^n$ for $\alpha \in[0,1]$.

\item Type I: If $n=s, d=1, q>n/2$, then $f_{\alpha}(t)=t^n-(1-\alpha) t^{n-q}-\alpha$ for $\alpha \in[0,1]$.

\item Type II: If $n=qd, d>1$, then $f_{\alpha}(t)=(t^q+\alpha-1)^{d}-\alpha^{d}t^{z}$, where $z=qd-s$ and $z \in \{1,\ldots,q-1\}$ for $\alpha \in[0,1]$.

\item Type III: If $n=s, d>1$, then $f_{\alpha}(t)=t^{y}(t^q+\alpha-1)^{d}-\alpha^d$, where $y=s-qd$ and $y \in \{1,\ldots,q-1\}$ for $\alpha \in[0,1]$. 
\end{itemize}

Next, we recall results from  \cite{kirklandsmigoc} that characterise the sparsest realising matrices for Type I, II, and III reduced Ito polynomials, through their associated digraphs.

\begin{theorem}[\cite{kirklandsmigoc}, Type I] \label{thm:TI_realisation}
Let $q,n \in \N$, $\gcd(q,n)=1$ and $2q>n>q$. 
Then for a stochastic matrix $A$, the following statements are equivalent:
\begin{enumerate}
\item $A \in \mc M_n^0(q,s).$ 
\item $\Gamma(A)$ is up to isomorphism equal to $\Gamma=C({\bf a}(n))+ \{(q, 1)\}$. 
\item $\Gamma(A)$ has one $n$-cycle, one $q$-cycle and no other cycles. 
\end{enumerate}
In addition, if the weight on the edge $(q,1)$ in $\Gamma(A)$ is equal to $(1-\alpha)$, then the characteristic polynomial of $A$ is $f_{\alpha}(t)=(t+\alpha-1)^n-\alpha^n$ .
\end{theorem}

\begin{theorem}[\cite{kirklandsmigoc}, Type II] \label{thm:TII_realisation}
Let $n=qd$, $s=qd-z$ where $z \in \{1,\ldots, q-1\}$, $\gcd(q,s)=1$, $d\geq 2$. For a stochastic matrix $A$, the following statements are equivalent:
\begin{enumerate}
    \item $A \in \mc M_n^0(q,s).$
    \item $\Gamma(A)$ is up to isomorphism equal to 
    \begin{equation}\label{eq:TIIdiagraph}
    \Gamma=\cup_{i=1}^d \left(C\left(q(i-1)\cdot{\bf e}+{\bf a}(q)\right)+\left \{(iq-z_i,\langle 1+iq\rangle_n)\right \}\right),
    \end{equation}
    where $z_i$ correspond to some (ordered) partition of $z$ into $d$ parts: $z=z_1+\ldots+z_d$, $z_i\geq 0$. Furthermore, all the edges $(iq-z_i,\langle 1+iq\rangle_n),i=1,\ldots,d$, have equal weight $\alpha$, for some $\alpha \in (0,1)$.   

    \item $\Gamma(A)$ has $d$ $q$-cycles, one $s$-cycle, and no other cycles. All $q$-cycles have equal weight.
\end{enumerate}

For $A$ to have the characteristic polynomial $f_{\alpha}(t)=(t^q+\alpha-1)^{d}-\alpha^{d}t^{z}$, the weights on all the edges $(iq-z_i,\langle 1+iq\rangle_n)$,  $i=1\ldots,d$,  have to be equal to $\alpha$.
\end{theorem}

In the second item of the above theorem, a detailed description of the digraph $\Gamma(A)$ is given, where $z_i$ count the number of vertices on the $i$-th $q$-cycle  $C(q(i-1)\cdot{\bf e}+{\bf a}(q))$ that are not included in the (unique) $s$-cycle in $\Gamma(A)$. Different partitions of $z$ will result in different graphs $\Gamma(A)$. To identify the partitions that produce non-isomorphic graphs we offer the following definition. 
\begin{definition}\label{def:TIIpartition class}
Let $A \in \mc M_n^0(q,s)$, $n=qd$, $s=qd-z$, where $\Gamma(A)$ is isomorphic to a directed graph of the form \eqref{eq:TIIdiagraph} for the partition $z=z_1+\ldots+z_d$. Let $\bf z=\npmatrix{z_1 & \ldots & z_d}$. We say that $\T(\bf z)$ is \emph{the partition class} of $A$, denoted by $\P2(A)$.

\end{definition}

Note that the partition class $\T(\bf z)$ determines $A \in \mc M_n^0(q,s)$ up to permutation similarity, and the associated directed graph up to isomorphism.

\begin{theorem}[\cite{kirklandsmigoc}, Type III] \label{thm:TIII_realisation}
Let $n=qd+y$, where $\gcd(q,n)=1$, $d\geq 2$, and $y \in \{1,\ldots, q-1\}$. For a stochastic matrix $A$, the following statements are equivalent:
\begin{enumerate}
    \item $A \in \mc M_n^0(q,s).$
    \item\label{thm:TIII,3} $\Gamma(A)$ is up to isomorphism equal to 
    \begin{equation}\label{eq:TIIIdiagraph}
    \Gamma=C\left({\bf a}\left (n \right )\right) +  \left\{\left( iq+\sum_{k=1}^{i} y_k, 1+ (i-1) q+\sum_{k=1}^{i}y_k\right);i=1,\ldots,d\right\},
    \end{equation}
     where $y_i$ correspond to an (ordered) partition of $y$ into $d$ parts: $y=y_1+\ldots+y_d$, $y_i\geq 0.$ In addition, the edges $( iq+\sum_{k=1}^{i} y_k, 1+ (i-1) q+\sum_{k=1}^{i}y_k),i=1,\ldots,d$, all have the same weight.
    \item $\Gamma(A)$ has $d$ $q$-cycles, one $n$-cycle, and no other cycles, where the weights on each of the $q$-cycles are equal. 
\end{enumerate}
For $A$ to have the characteristic polynomial $f_{\alpha}(t)=t^{y}(t^q+\alpha-1)^{d}-\alpha^d$, the weights on the edges $( iq+\sum_{k=1}^{i} y_k, 1+ (i-1) q+\sum_{k=1}^{i}y_k)$, $i=1,\ldots,d$, have to be equal to $1-\alpha$.
\end{theorem}

In the theorem above, a digraph $\Gamma$ consists of $d$ $q$-cycles and an $s$-cycle that contains all the vertices of the $q$-cycles together with paths that connect the $q$-cycles. The partition of $y$ determines the number of vertices on those connecting paths. To recognize partitions that produce non-isomorphic graphs described in item \ref{thm:TIII,3} of the theorem above, we introduce the following definition. 

\begin{definition}\label{def:TIIIpartition class}
Let $A \in \mc M_n^0(q,s)$, $n=qd+y$, $\gcd(q,n)=1$, $d\geq 2$, where $\Gamma(A)$ is isomorphic to a directed graph of the form \eqref{eq:TIIIdiagraph} for the partition $y=y_1+\ldots+y_d$. Let $\bf y=\npmatrix{y_1 & \ldots & y_d}$. We say that $\T(\bf y)$ is \emph{the partition class} of $A$, denoted by $\PIII(A)$. 
\end{definition}

To summarise, in all cases (Type I, II, and III) the sparsest realising matrices for the arc $\mc K_n(q,s)$ are completely described by their digraphs and the weights on $q$-cycles in the digraphs. Moreover, the digraph for Type I is unique, for Type II and III the digraphs are associated with the partitions of $z$ and $y$, respectively.

\subsection{Power of a single cycle}

The following well-known lemma on the powers of a cycle will be needed to study the powers of digraphs associated with stochastic matrices. A short proof is given for completeness.

\begin{lemma}\label{lemma:powerofcycle} 
Let $C({\bf a}(k))$ be a cycle with $k$ vertices, $c,h\in \N$ so that $\gcd(c,k)=h$. Furthermore, let $k=k_1h$, $c=c_1h$, $c_1,k_1 \in \N$. Then $C({\bf a}(k))^{(c)}$ is a digraph with $h$ cycles of order $k_1$, i.e $C({\bf a}(k))^{(c)}=\cup_{i=1}^h C_i$, where 
\begin{align}\label{eq:cycle power}
    C_i= C\left (i \cdot {\bf e}+\langle c \cdot {\bf a}_0(k_1-1) \rangle_k \right).
\end{align}
\end{lemma}

\begin{proof} 
Since any vertex $i \in V(C({\bf a}(k)))$ has a unique vertex at a distance $c$ in $C({\bf a}(k))$, there is a unique edge outgoing from $i$: $(i, {\langle i+ c\rangle_k}) \in  E(C({\bf a}(k))^{(c)})$. Furthermore, for any $i \in \{1,2,\ldots,k\}$, the vertices ${\langle i+\ell c\rangle_k}$, $\ell \in \{0,\ldots,k_1-1\}$, form a cycle of order $k_1$ in $C({\bf a}(k))^{(c)}$. This follows from $\langle i +k_1c \rangle_k=\langle i \rangle _k$  and $\langle i +\ell c \rangle_k \neq \langle i \rangle _k$ for any $\ell<k_1$. Finally, $k=k_1 h$ implies that there are $h$ cycles of order $k_1$ in $C({\bf a}(k))^{(c)}$. \end{proof}

The next remark considers two special cases of the above lemma, that we will encounter in the upcoming sections.

\begin{remark}\label{remark:powerofcycle}
We consider the $c$-th strong power of $C({\bf a}(k))$ in the case when $c$ divides $k$ and in the case when $\gcd(c,k)=1$. 
    \begin{enumerate}
        \item If $k=k_1 c$ then $C({\bf a}(k))^{(c)}=\cup_{i=1}^c C_i$, where 
        $C_i= C(i \cdot {\bf e}+\langle c \cdot {\bf a}_0(k_1-1) \rangle_k).$
        
        \item If $\gcd(k,c)=1$, then $C({\bf a}(k))^{(c)}=C(\langle c\cdot {\bf a}(k)\rangle_k)$.   
   \end{enumerate}
\end{remark}

\section{Powers of Sparsest Realising Matrices}\label{sec:powers_of_matrices}

\subsection{Type II arc is a power of a Type I arc}

In this subsection we assume $n=dq$, $\gcd(qd,z)=1$, $z \in \{1,\ldots,q-1\}$. From Theorem \ref{thm:powers} and Corollary \ref{cor:powers1} we have: $$\mc K_{dq}(dq-z,qd)^d=\mc K_{dq}(q,qd-z).$$
The theorem below determines the partition class of $B^d$ for $B \in \mc M_n^0(dq-z,qd)$.

\begin{theorem}\label{thm:TI-TII}
Let $B \in \mc M_n^0(dq-z,qd)$, where $d\geq 2$, $n=qd$, $\gcd(qd,z)=1$ and $z \in \{1,\ldots,q-1\}$.
We define $\beta:=d -\langle z\rangle_d$, $w:=\frac{z -\langle z\rangle_d}{d}$, and for $j=1,\ldots,d$:
$$z(j):=\begin{cases}
w &\text{ for } j=1,\ldots,\beta \\ 
w+1  &\text{ for } j=\beta+1,\ldots,d.
\end{cases}$$
The elements of the partition class $\PII(B^d)$ consist of the parts $z(j)$, $j=1,\ldots,d$, where the part $z(j)$ is followed by the part $z(\langle j-\beta \rangle_d)$ in the partition.
\end{theorem}

\begin{proof}
    Let $B \in  \mc M_n^0(\hat q, \hat s)$, where $\hat q=s=qd-z$, $\hat s=qd$, $\gcd(\hat q,n)=1$ and $2\hat q>n>\hat q$.

 \quad
 
\noindent{\bf Digraph.} By Theorem \ref{thm:TI_realisation}, $\Gamma(B)$ is isomorphic to $\widehat \Gamma=C({\bf{a}}(n))+\{(s,1)\}$. First, we want to find a digraph $\Gamma$ that is isomorphic to $\widehat \Gamma^{(d)}$.

Taking $k=n$, $c=d$ and $k_1=q$ in the first part of Remark \ref{remark:powerofcycle}, we get $C({\bf a}(n))^{(d)}=C({\bf a}(qd))^{(d)}=\cup_{j=1}^d C_j$, where

\begin{align}\label{eq:q-cycles eq}
   C_j:=C(j\cdot{\bf e}+d\cdot {\bf a}_0(q-1))
\end{align}
is a $q$-cycle in $\widehat \Gamma^{(d)}$.
Furthermore, the edge $\hat e=(s,1)$ in $\widehat \Gamma$ contributes the following edges in $\widehat \Gamma^{(d)}$:
\begin{align}\label{eq:TI-TIIconnecting edges}
    e_t:=(s-t+1,1+d-t),t=1,\ldots,d,
\end{align}
where we define $e_t:=(v_O(t),v_I(t))$ for later use. With this, we have determined $\widehat \Gamma^{(d)}$ to be:
$$\widehat \Gamma^{(d)}=\cup_{j=1}^d C_j+ \{e_t:t=1,\ldots,d\}.$$
In addition, if the edge $\hat e$ has weight $1-\alpha$ in $\widehat \Gamma$, then all the $q$-cycles, $C_j,j=1,\ldots,d$, have equal  weight $\alpha$ in $\widehat \Gamma^{(d)}$.

 Finally, note that the edges $e_t,t=1,\ldots,d$, connect the $q$-cycles  $C_j,j=1,\ldots,d$, to form an $s$-cycle in $\widehat \Gamma^{(d)}$. The $s$-cycle in $\widehat \Gamma^{(d)}$ consists of all the edges $e_t$ and certain paths that are subgraphs of $C_j$'s. The lengths of those paths will help us to determine the partition class of $B^d$. 

 \quad

\noindent{\bf The ordering of parts in the partition.}  
Let $z(j)$ be the number of vertices in $V(C_j)$ that do not belong to the $s$-cycle in $\widehat \Gamma^{(d)}$. Note that the elements of $\PII(B^d)$ consist of  $z(j)$ in some order. 

To determine the order of $z(j)$ in $\PII(B^d)$ we fix $j\in \{1,\ldots,d\}$ and note that there exists a unique $t'\in \{1,\ldots,d\}$ such that $v_I(t') \in V(C_j)$. We say that $e_{t'}$ is the incoming edge for $C_j$. From $v_I(t')=1+d-t'$ and $V(C_j)=\{j+\ell d:\ell=0,\ldots,q-1\}$, we conclude that if $e_{t'}$ is the incoming edge for our fixed $C_j$, then $j+t'$ is congruent $1$ modulo $d$. Thus, $e_{\langle 1-j \rangle_d}$ is the incoming edge for $C_j$. Similarly, there exists a unique $t$ such that $v_O(t) \in V(C_j)$ and we say that $e_t$ is the outgoing edge from $C_j$. For this to be true, we must have $j+t$ congruent to $1-z$ modulo $d$. This relation, and the fact that $t\in\{1,\ldots, d\}$, uniquely define $t$ to be:
\begin{align}\label{eq:j-values}
t:=\begin{cases}
\beta+1-j, &\text{ for }j=1,\ldots,\beta,\\
\beta+1+d-j, &\text{ for }j=\beta+1,\ldots,d,
\end{cases}    
\end{align}
where $\beta=d-\langle z \rangle_d$. For this $t$ we get:
$$v_I(t)=\begin{cases}
d+(j-\beta), &\text{ for }j=1,\ldots,\beta,\\
j-\beta,  &\text{ for }j=\beta+1,\ldots,d,
\end{cases}
$$
 from \eqref{eq:TI-TIIconnecting edges} and \eqref{eq:j-values}. In particular, $v_I(t) \in C_{\langle j-\beta \rangle_d}$ and 
 $C_{j}$ is connected to $C_{\langle j-\beta \rangle_{d}}$ in the $s$-cycle in $\widehat \Gamma^{(d)}$. Equivalently, $z(j)$ is followed by $z(\langle j-\beta\rangle_{d})$ in the partition class $\PII(B^d)$.

\quad

\noindent{\bf Parts of the partition.} To determine the parts that appear in the partition class of $B^d$, we need to determine the lengths of the paths that are intersections of the $q$-cycles $C_j$ and the $s$-cycle in $\widehat \Gamma^{(d)}$. In other words, the number of vertices in such a path is the same as the number of vertices $C_j$ is contributing to the $s$-cycle in $\widehat \Gamma^{(d)}$. It is clear from the discussion so far that $v_I(\langle 1-j \rangle_d)$ is the first and $v_O(\langle 1-z-j \rangle_d)$ is the last vertex on this path. From \eqref{eq:TI-TIIconnecting edges} and \eqref{eq:j-values} we get $v_I(\langle j-1 \rangle_d)=j$ and
\begin{align*}
v_O(t)&=
\begin{cases}
qd-z-\beta+j, &\text{ for }j=1,\ldots,\beta\\
qd-z-\beta-d+j, &\text{ for }j=\beta+1,\ldots,d,
\end{cases}\\
&=
\begin{cases}
(q-w-1)d+j, &\text{ for }j=1,\ldots,\beta\\
(q-w-2)d+j, &\text{ for }j=\beta+1,\ldots,d,
\end{cases}
\end{align*}
where $\beta=d-\langle z \rangle_d$ and $w:=\frac{z-\langle z \rangle_d}{d}$.

Let $k(j)$ denote the number of vertices from $C_j$ that are contained in the $s$-cycle.
Since vertices in $C_j$ are consecutively numbered by $j+\ell d, \ell=0,\ldots,q-1$, we have:
$$k(j)=\begin{cases}
q-w, &\text{ for }j=1,\ldots,\beta\\
q-w-1, &\text{ for }j=\beta+1,\ldots,d,
\end{cases}$$
or equivalently, $z(j):=q-k(j)$ is the number of vertices from $C_{j}$ that are not on the $s$-cycle. With this, we have determined the numbers that appear in the partition class $\PII(B^d)$. 
\end{proof}

\begin{remark}\label{remark:TI-TII}
For $B \in \mc M_n^0(qd-z,qd)$, let us define the row vector ${\bf z}$ consisting of parts of the partition class $\PII(B^d)$:
 ${\bf z}:=\npmatrix{z(1) & \ldots & z(d)}$ as in Theorem \ref{thm:TI-TII}. Since, $z(j)$ is followed by $z(\langle j-\beta \rangle_{d})$ in the partition class $\PII(B^d)$, we need to permute the elements of ${\bf z}$ to get the row vector:
 $${\bf z'}:=\npmatrix{z(1) & z(\langle 1-\beta \rangle_d) & \ldots & z(\langle 1-(d-1)\beta) \rangle_d}.$$
which belongs to the partition class $\PII(B^d)$. 
\end{remark}

\begin{example}
 For $n=120$ we have  $\mc K_n(15,109)=\mc K_n(109,120)^8$. Below we follow the steps of the proof of Theorem \ref{thm:TI-TII} for $B\in \mc M_n^0(109,120)$. We assume the notation developed in the proof. In particular, $q=15$, $s=109$, $d=8$, $\hat q=s=109$, $\hat s=qd=120$ and $z=qd-s=11$.
\quad

 \noindent{{\bf Digraph.} } Let the digraph of $B$ be  $\widehat \Gamma=C({\bf a}(120))+\{(109,1)\}$. Then  $\widehat \Gamma^{(8)}$ consists of $8$ cycles of order $15$: 
 $C(j\cdot {\bf e}+8 \cdot {\bf a}_0(15))$,  $j=1,\ldots,8$, and additional edges $\{(110-t,9-t):t=1,\ldots,8\}$ that connect the $15$-cycles to form a $109$-cycle in $\widehat \Gamma^{(8)}$. Note that the weights on the edges $(110-t,9-t)$, $t=1,\ldots,8$, in $\widehat \Gamma^{(8)}$ are the same as the weight on the edge $(109,1)$ in $\widehat \Gamma$.

\quad

 \noindent{\bf Partition.} 
From $\beta=5$ and $w=1$ we determine
the parts of the partition class $\PII(B^8)$:
 $$z(j)=\begin{cases}
1 &\text{ for } j=1,\ldots,5 \\ 
2 &\text{ for } j=6,7,8. 
\end{cases}$$
The vector ${\bf z'}$ defined in Remark \ref{remark:TI-TII} is:
\begin{align*}
 {\bf z'} &=\npmatrix{  z(1) & z(4) & z(7) & z(2) & z(5) & z(8) & z(3) & z(6)}  \\
 &=\npmatrix{
1 & 1 & 2 & 1 & 1 & 2 & 1 & 2 \\
},
\end{align*}
and $\PII(B^8)=\T({\bf z'})$.
\end{example}

\begin{corollary}\label{cor:TI-TII}
Let $A \in \mc M_n^0(q,qd-z)$ have the associated partition $\npmatrix{z_1 & \ldots & z_d}\in ~\PII(A)$. Then $A=B^d$ for some $B \in \mc M_n^0(qd-z,qd)$ if and only if:
\begin{enumerate}
    \item There exists $w$ such that $z_j \in \{w,w+1\}$. (If all $z_j$ are equal then we say they are all equal to $w+1$). 
    
    We define $\beta:=|\{j\in \{1\ldots,d\}; z_j=w\}|$ and :
    $$z(j)=\begin{cases}
w &\text{ for } j=1,\ldots,\beta \\ 
w+1  &\text{ for } j=\beta+1,\ldots,d.
\end{cases}$$

\item The partition class $\PII(A)$ consists of the parts $z(j)$, $j=1,\ldots,d$. Further, the part $z(j)$ is followed by the part $z(\langle j-\beta \rangle_d)$ in $\PII(A)$.
\end{enumerate}
\end{corollary}

\begin{proof}
Let $A \in \mc M_n^0(q,qd-z)$ with associated partition $\npmatrix{z_1 & \ldots & z_d}\in \PII(A)$ and $B \in \mc M_n^0(qd-z,qd)$. Then $A=B^d$ if and only if the parts $z_j$ correspond to the parts $z(j)$ in Theorem \ref{thm:TI-TII} which is true if and only if both the conditions of the corollary are satisfied.
\end{proof}

\begin{example}
Let $n=120$, then $\mc K_n(15,107)=\mc K_n(107,120)^8$, where $q=15$, $s=107$, $d=8$ and $z=13$. Recall,  by Definition \ref{def:TIIpartition class}, that any $8$ non-negative integers that sum to $13$ result in a partition class for a matrix in $\mc M_n^0(15,107)$.
So, let $A \in \mc M_n^0(15,107)$ and $\bf z=\npmatrix{z_1 & \ldots & z_{8}} \in \PII(A)$.

If $A=B^8$ for some $B\in \mc M_n^0(15,120)$, then the elements in ${\bf z'}_A$ must be in the set $\{w,w+1\}$ for some $w \in \{1,2,\ldots,13\}$. If ${\bf z'}_A$ either contains three different numbers or numbers that are more than one apart, then $A$ will not be a power of any matrix $B \in \mc M_n^0(107,120)$. Thus, the necessary condition from the first part of the corollary leaves us just one choice for the vector ${\bf z}$ defined in Remark \ref{remark:TI-TII}:
$${\bf z}=\npmatrix{1 & 1 & 1 & 2 & 2 & 2 & 2 & 2}.$$
The above ${\bf z}$ gives $\beta=3$ and $w=1$. The entries of ${\bf z}$ can be permuted in several ways, but ${\bf z}$ gives us the unique partition class for $\PII(A)$ as defined in Corollary \ref{cor:TI-TII}. Therefore $A=B^8$ for some $B \in \mc M_n^0(107,120)$ if and only if ${\bf z'}\in \PII(A)$, where:
\begin{align*}
   {\bf z'}&=\npmatrix{z(1) & z(6) & z(3) & z(8) & z(5) & z(2) & z(7) & z(4)}\\
   &=\npmatrix{1 & 2 & 1 & 2 & 2 & 1 & 2 & 2}.
\end{align*}
\end{example}

\subsection{Type II arc is a power of a Type II arc}

Throughout this subsection we assume: $n=c\hat{d}q$, $ z \in \{1,\ldots,q-1\}$, $\gcd(cq,z)=1$. By Theorem \ref{thm:powers} we have: $$\mc K_{n}(cq,c\hat dq-z )^c=\mc K_{n}(q,c \hat dq-z).$$ 
Given the partition class $\PII(B)$  for  $B \in \mc M_{n}^0(cq,c\hat dq-z )$ we want to determine the partition class $\PII(B^c)$.

\begin{theorem}\label{thm:TII-TII}
Let $B \in \mc M_{n}^0(cq,c\hat dq-z )$ and $\npmatrix{\hat z_1 & \ldots & \hat z_{\hat d}} \in \PII(B)$. For $i=1\ldots, \hat d$ and $j=1\ldots, c$, we define $\beta_i=c-\langle\hat z_i\rangle_c$, $w_i=\frac{\hat z_i-\langle\hat z_i\rangle_c}{c}$, and $$z(i,j)=\begin{cases}
w_i &\text{ for } j=1,\ldots,\beta_i \\ 
w_{i}+1  &\text{ for } j=\beta_{i}+1,\ldots,c.
\end{cases}$$
The partition class $\PII(B^c)$ is defined as follows:
\begin{itemize}
\item the parts of the partition are equal to $z(i,j)$, $i=1,\ldots,\hat d$, $j=1\ldots,c$, 
\item in the partition $z(i,j)$ is followed by $z(\langle i+1 \rangle_{\hat d},\langle j-\beta_i \rangle_c)$. 
 \end{itemize}
\end{theorem}

\begin{proof} 
Let $B \in \mc M_n^0(\hat q, \hat s)$, where $\hat q=cq$, $\hat s=c\hat d q-z $. 

\quad

\noindent{\bf Digraph of $\bf B^c$.} Given a directed graph $\widehat \Gamma$ that is isomporhic to $\Gamma(B)$, we first find a directed graph $\Gamma=\widehat\Gamma^{(c)}$ that is isomorphic to $\Gamma(B^c)$. By Theorem \ref{thm:TII_realisation}, $\Gamma(B)$ is isomorphic to
        $$\widehat \Gamma=\cup_{i=1}^{\hat d} \left ( C(\hat q(i-1)\cdot{\bf e}+{\bf a}(\hat q))+\{(i \hat q-\hat z_i,\langle 1+i\hat q\rangle_n)\}\right),$$
        where $\npmatrix{\hat z_1&\ldots & \hat z_{\hat d}} \in \PII(B)$.
  Denoting the $\hat q$-cycles in $\widehat \Gamma$ by $\widehat C_i:=C(\hat q(i-1)\cdot e+{\bf a}(\hat q))$, $i=1,\ldots,\hat d$,  and taking $k=\hat q$, $k_1=q$ in the first part of Remark \ref{remark:powerofcycle}, we get
$\widehat C_i^{(c)}=\cup_{j=1}^c C_{i,j},$ where for $i=1,\ldots,\hat d$ and $j=1,\ldots,c$, we denote: 
\begin{align}\label{eq:cyclesTII}
  C_{i,j}:=C \left( (qc(i-1)+j)\cdot {\bf e}+c\cdot {\bf a}_0(q-1)\right ).
\end{align}
Note that each $C_{i,j}$ is a $q$-cycle in $\widehat \Gamma^{(c)}$. 

Next, we consider contribution of the edges $\hat e_i=(i \hat q-\hat z_i,\langle 1+i\hat q\rangle_n), i=1,\ldots,\hat d$, to $\widehat\Gamma^{(c)}$. Note that the edge $\hat e_i$ connects $\widehat C_i$ to $\widehat C_{\langle i+1 \rangle_{\hat d}}$ in $\widehat \Gamma$, and contributes the edges $\{e_{i,t}:t=1,\ldots,c\}$ to  $\widehat \Gamma^{(c)}$, where:
\begin{align}\label{eq:connecting_edgesTII}
    e_{i,t}:=(\langle i\hat q-\hat z_i-(t-1)\rangle_n,\langle i \hat q+c-(t-1)\rangle_n).
\end{align}
Let us denote $e_{i,t}:=(v_O(i,t),v_I(i,t))$ for future use. Note that  in $\widehat \Gamma^{(c)}$, the edges $e_{i,t}$  connect  the $q$-cycles to form an $s$-cycle. This $s$-cycle is the union of the edges $e_{i,t}$ and paths that are part of $C_{i,j}$'s. These paths are discussed in the third section of the proof where we talk about the partition parts.
At this point we can write down $\widehat \Gamma^{(c)}$ as:
$$\widehat \Gamma^{(c)}=\cup_{i=1}^{\hat d}\left(\cup_{j=1}^c C_{ij}+\{e_{i,t}:t=1,\ldots,c\}\right).$$
The weight $\alpha$ on the edges $\hat e_i,i=1,\ldots,\hat d$, in $\widehat \Gamma$ implies the weight $\alpha$ on the edges $e_{i,t},i=1,\ldots,\hat d,t=1,\ldots,c$ in $\widehat \Gamma^{(c)}$. Thus, the weights on all the $q$-cycles, $C_{i,j},i=1,\ldots,\hat d, j=1,\ldots,c$, are equal to $1-\alpha$ in $\widehat \Gamma^{(c)}$.

\quad

\noindent{\bf Connecting the $q$-cycles and ordering the parts in the partition.} To determine $\PII(B^c)$ we take a closer look at how the edges $e_{i,t}$ connect the cycles $C_{i,j}$. In particular, let $z(i,j)$ be the number of vertices in $V(C_{i,j})$ that are not on the $s$-cycle in $\widehat \Gamma^{(c)}$. To determine the ordering of $z(i,j)$ in $\PII(B^c)$ we fix $i$ and $j$ and determine how the $q$-cycles follow each other to form the $s$-cycle in $\widehat \Gamma^{(c)}$.

For each pair $i,j$, $i\in \{1,\ldots,\hat d\}$ and $j\in\{1,\ldots,c\}$, there exists precisely one $t'\in \{1,\ldots,c\}$ so that $v_I(i-1,t') \in V(C_{i,j})$. We say that $e_{i-1,t'}$ is \emph{an incoming edge} for $C_{i,j}$. From $v_I(i-1,t')=\langle (i-1) \hat q+c-t'+1)\rangle_n$ and  $V(C_{i,j})=\{(qc(i-1)+\ell c+j),\ell=0,\ldots,q-1\}$, we deduce that $j+t'$ is congruent to $1$ modulo $c$. In short, $e_{i-1,\langle 1-j \rangle_c}$ is the incoming edge for $C_{i,j}$. Similarly, there exists a unique $t\in\{1,\ldots,c\}$ such that $v_O(i,t)\in V(C_{i,j})$ and we say that $e_{i,t}$ is an \emph{outgoing edge} for $C_{i,j}$. This implies that $j+t$ is congruent to $1-\hat z_i$ modulo $c$. This relation and the fact that $t\in \{1,\ldots,c\}$ uniquely define $t$ to be:
\begin{align}\label{eq:t-values}
  t:=\begin{cases}
\beta_i+1-j &\text{ for } j=1,\ldots,\beta_i \\ 
\beta_{i}+1+c-j  &\text{ for } j=\beta_{i}+1,\ldots,c, 
\end{cases}
\end{align}
where $\beta_i:=c-\langle \hat z_i \rangle_c$. Finally, to determine how the cycles $C_{i,j}$ are ordered to form the $s$-cycle in $\widehat \Gamma^{(c)}$, we use equation \eqref{eq:connecting_edgesTII} and equation \eqref{eq:t-values} to get:
$$v_I(i,t)=\begin{cases}
i\hat q+c+j-\beta_i &\text{ for } j=1,\ldots,\beta_i \\ 
i\hat q+j-\beta_i  &\text{ for } j=\beta_{i}+1,\ldots,c.
\end{cases}$$
This implies $v_I(i,t)\in C_{\langle i+1 \rangle_{\hat d}, \langle j-\beta_i \rangle_c}$. Consequently, $z(i,j)$ is followed by $z(\langle i+ 1 \rangle_{\hat d},\langle j-\beta_i \rangle_c)$ in $\PII(B^c)$.

\quad

\noindent{\bf Parts of the partition.} To determine the parts that appear in $\PII(B^c)$, we want to determine how many vertices from each $C_{i,j}$ cycle are (are not) contained on the $s$-cycle in $\widehat\Gamma^{(c)}$. Equivalently, we want to know the number of vertices on the path that is the intersection between $C_{i,j}$  and the $s$-cycle in $\widehat\Gamma^{(c)}$. Let us denote this number by $k(i,j)$. From the above discussion, we know that $v_I(i-1,\langle 1-j \rangle_c)$ is the first and $v_O(i,t)$ is the last vertex on this path for $t=\langle 1-j-\hat z_i \rangle_c$. Using equation \eqref{eq:connecting_edgesTII} and equation \eqref{eq:t-values}, we get:
\begin{align*}
v_I(i-1,\langle 1-j \rangle_c&=\hat q(i-1)+j,\\
v_O(i,t)&=
\begin{cases}
i\hat q-\hat z_i-\beta_i+j, &\text{ for }j=1,\ldots,\beta_i\\
i\hat q-\hat z_i-c-\beta_i+j &\text{ for }j=\beta_i+1,\ldots,c.
\end{cases}
\end{align*}
From here, we can write:
\begin{align*}
 v_O(i,t)&=
\begin{cases}
(i-1)\hat q+(q-w_i-1)c+j, &\text{ for }j=1,\ldots,\beta_i\\
(i-1)\hat q+(q-w_i-2)c+j &\text{ for }j=\beta_i+1,\ldots,c,
\end{cases}
\end{align*}
where $\beta_i=c-\langle\hat z_i\rangle_c$, and $w_i=\frac{\hat z_i-\langle\hat z_i\rangle_c}{c}$.
Recalling that the vertices of $C_{i,j}$ are consecutively numbered by $(qc(i-1)+\ell c+j)$, $\ell=0,\ldots,q-1$, we conclude that:
$$k(i,j)=\begin{cases}
q-w_i, &\text{ for }j=1,\ldots,\beta_i\\
q-w_i-1, &\text{ for }j=\beta_i+1,\ldots,c,
\end{cases}$$
or equivalently, $z(i,j):=q-k(i,j)$ is the number of vertices from $C_{i,j}$ that are not on the $s$-cycle. With this, we have determined the numbers that appear in the partition class $\PII(B^c)$. 

\end{proof}

\begin{remark}\label{remark:TII-TII}
Given $\hat{\mathbf{z}} \in \mc{P}(B)$ we can define a matrix $Z'$ that satisfies $\vecm(Z')^T \in ~\PII(B^c)$ as follows. 
From $\hat{\mathbf{z}}$ we define the parts $z(i,j)$ as in the statement of Theorem \ref{thm:TII-TII}, and consider the matrix:
$$Z:=\left(
\begin{array}{cccc}
 z(1,1) & z(1,2) & \ldots & z(1,c) \\
 \vdots & \vdots & \ddots & \vdots \\
 z(\hat d,1) & z(\hat d,2) & \ldots & z(\hat d, c)\\
\end{array}
\right).$$
The elements of $Z$ are the same as the parts of the partitions in  $\PII(B^c)$ (with multiplicities). Next, we determine their ordering in the partition. 
Using the fact that $z(i,j)$ is followed by $z(\langle i+1 \rangle_{\hat d}, \langle j-\beta_i \rangle_c)$, we define the $j$-th column, $z'_{\star,j}$, $j=1,\ldots,c$, of the matrix $Z'$ to be: 

$$z'_{\star,j}:=\left(
\begin{array}{c}
  z(1,\langle 1-(j-1)\beta \rangle_c) \\
 
  \vdots \\
 
 z(i, \langle 1-(j-1)\beta-\sum_{k=1}^{i-1}\beta_k \rangle_c)  \\
 
  \vdots  \\

  z(\hat d, \langle 1-(j-1)\beta-\sum_{k=1}^{\hat d-1}\beta_k \rangle_c)  \\
 
\end{array}
\right).$$
The partition class $\PII(B^c)$ is equal to $\T(\bf z)$, where ${\bf z}=\vecm(Z')^T$. Note that the unordered multiset of elements in the $j$-th row of $Z$ is equal to the unordered multiset of elements in the $j$-th row of $Z'$ but the elements appear in matrices $Z$ and $Z'$ in different orders. 
\end{remark}

\begin{example}\label{Ex:T2-T2}
Let $n=120$, then $\mc K_n(6,115)=\mc K_n(24,115)^4$, where $q=6$, $c=4$, $qc=\hat q=24$, $\hat s=115$, $\hat d=5$, $z=5$ and $z=\hat z_1+\hat z_2+\hat z_3+\hat z_4+\hat z_5$.

Let $B\in \mc M_n^0(24,115)$. In this example, we consider a few possibilities for partitions ${\bf \hat z}  \in \PII(B)$. In each case, we illustrate parts of the proof and write down matrices $Z$ and $Z'$ defined in Remark \ref{remark:TII-TII}.
\begin{enumerate}
    \item $\npmatrix{\hat z_1 & \hat z_2 & \hat z_3 & \hat z_4 & \hat z_{5}}=\npmatrix{5 & 0 & 0 & 0 & 0}.$

\quad

\noindent{\bf Digraph of $\bf B^4$.} $\widehat \Gamma$ consists of $5$ cycles of order $24$, $C(24(i-1)\cdot {\bf e}+ {\bf a}(24))$, $i=1,\ldots,5$, and the edge set, $\{(19,25),(48,49),(72,73),(96,97),(120,1)\}$ that connect the $24$-cycles to make the $115$-cycle. 
    
$\widehat \Gamma^{(4)}$ consists of $20$ cycles of order $6$, $C \left( (24(i-1)+j)\cdot {\bf e}+c\cdot {\bf a}_0(6)\right )$ for $i=1,\ldots,5$, $j=1,\ldots,4$, and the edge set:

$$(\langle 24 i-\hat z_i-(t-1)\rangle_n,\langle 24i +4-(t-1)\rangle_n)$$ for $i=1,\ldots,5$, $t=1,\ldots,4$, that connect the $6$-cycles to form a cycle of order $115$. Note that each of these connecting edges has the weight $\alpha$ in $\widehat \Gamma^{(4)}$ if the connecting edges in $\widehat \Gamma$ have the weight $\alpha$.
   
\quad

\noindent{\bf Parts of the Partition.}  
From $\hat z_1=5$ we get $\beta_1=3$ and $w_1=1$. Similarly, $ \hat z_2=\hat z_3=\hat z_4=\hat z_5=0$ give $\beta_2=\beta_3=\beta_4=\beta_5=0$ and $w_2=w_3=w_4=w_5=-1$. 
    
Thus, the matrix $Z$ defined in Remark \ref{remark:TII-TII} is equal to: $$Z=\left(
\begin{array}{cccc}
 1 & 1 & 1 & 2 \\
 0 & 0 & 0 & 0 \\
 0 & 0 & 0 & 0\\
 0 & 0 & 0 & 0\\
 0 & 0 & 0 & 0
\end{array}
\right).$$ 
At this point, we know that the elements in $\PII(B^c)$ will have three parts equal to $1$, one part equal to $2$, and all other parts equal to $0$. Since all the nonzero parts appear in the first row, we also know that every nonzero part will be followed by precisely four zeros. As this observation already uniquely defines the partition class, we move on to the next case.

\item $\npmatrix{\hat z_1 & \hat z_2 & \hat z_3 & \hat z_4 & \hat z_{5}}=\npmatrix{3 & 0 & 2 & 0 & 0}.$

{\bf Parts of the Partition.}
From $\hat z_1=3$ we get $\beta_1=1$, $w_1=0$. Next, $\hat z_3=2$ gives $\beta_3=2$, $w_3=0$. Finally, $\hat z_2=\hat z_4=\hat z_5=0$ gives  $\beta_2=\beta_4=\beta_5=0$ and $w_2=w_4=w_5=-1$.  
This implies: $$Z=\left(
\begin{array}{cccc}
 0 & 1 & 1 & 1 \\
 0 & 0 & 0 & 0 \\
 0 & 0 & 1 & 1\\
 0 & 0 & 0 & 0\\
 0 & 0 & 0 & 0
\end{array}
\right).$$
The elements of $\PII(B^c)$ have $5$ parts equal to $1$ and all other parts equal to $0$. 

{\bf Ordering.} This time the matrix $Z$ does not determine $\PII(B^c)$, and we also need:
$$Z'=\left(
\begin{array}{cccc}
 z(1,1) & z(1,2) & z(1,3) & z(1,4) \\
 z(2,4) & z(2,1) & z(2,2) & z(2,3) \\
 z(3,4) & z(3,1) & z(3,2) & z(3,3)\\
 z(4,2) & z(4,3) & z(4,4) & z(4,1)\\
 z(5,2) & z(5,3) & z(5,4) & z(5,1)
\end{array}
\right)=\left(
\begin{array}{cccc}
 0 & 1 & 1 & 1 \\
 0 & 0 & 0 & 0 \\
 1 & 0 & 0 & 1\\
 0 & 0 & 0 & 0\\
 0 & 0 & 0 & 0
\end{array}
\right).$$
Thus, the partition class $\PII(B^4)$ is given by $\T(\vecm(Z')^T)$.

\item Taking $\npmatrix{\hat z_1 & \hat z_2 & \hat z_3 & \hat z_4 & \hat z_{5}}=\npmatrix{2 & 0 & 2 & 0 & 1}$ we get:
$$Z=\left(
\begin{array}{cccc}
 0 & 0 & 1 & 1 \\
 0 & 0 & 0 & 0 \\
 0 & 0 & 1 & 1\\
 0 & 0 & 0 & 0\\
 0 & 0 & 0 & 1
\end{array}
\right)\text{ and } Z'=\left(
\begin{array}{cccc}
 0 & 0 & 1 & 1 \\
 0 & 0 & 0 & 0 \\
 1 & 1 & 0 & 0\\
 0 & 0 & 0 & 0\\
 0 & 0 & 0 & 1
\end{array}
\right),$$
and the partition class $\PII(B^c)$ is given by ${\bf z}=\T(\vecm(Z')^T)$.
\end{enumerate}
\end{example}

\begin{corollary}\label{cor:TII-TII}
Let $A \in \mc M_{n}^0(q,c\hat dq-z)$ and $\npmatrix{z_1 & \ldots z_{c \hat d}} \in \PII(A)$. Then $A =B^c$ for some  $B \in \mc M_{n}^0(cq,c\hat dq-z )$ if and only if:
\begin{enumerate}
    \item For every $i\in \{1,\ldots,\hat d\}$ there exists $w_i$ so that $z_{t \hat d+i}\in \{w_i,w_{i}+1\}$ for $t=0,\ldots,c-1$. (If for some $i$ all $z_{t \hat d+i}$ are equal, then we say that they are equal to $w_{i}+1$.) 
    
    For $i=1\ldots, \hat d$, we define, $\beta_i:=|\{t\in \{0\ldots,c-1\}; z_{t \hat d+i}=w_i\}|$, and
    $$z(i,j):=\begin{cases}
w_i &\text{ for } j=1,\ldots,\beta_i \\ 
w_{i}+1  &\text{ for } j=\beta_{i}+1,\ldots,c,
\end{cases}$$
where $i=1,\ldots,\hat d,j=1,\ldots,c.$

\item The partition class $\PII(A)$ consists of the parts $z(i,j)$, $i=1,\ldots,\hat d$, $j=1\ldots,c$. Further, in the partition $z(i,j)$ is followed by $z(\langle i+1 \rangle_{\hat d},\langle j-\beta_i \rangle_c)$. 
     
\end{enumerate}
If the conditions above are satisfied, then $A=B^c$ for $B \in \mc M_{n}^0(cq,c\hat dq-z )$ with $\PII(B)=\T(\npmatrix{\hat z_1 &\ldots &\hat z_{\hat d}}),$  where $\hat z_i:=c(w_i+1)-\beta_i$, for $i=1,\ldots,\hat d$.
\end{corollary}

\begin{proof}
Let $A \in \mc M_{n}^0(q,c\hat dq-z)$ and $\npmatrix{z_1 & \ldots & z_d} \in \PII(A)$. By Theorem \ref{thm:TII-TII}, $A=B^c$ for some $B \in \mc M_{n}^0(cq,c\hat dq-z )$ if and only if the items 1. and 2. in Corollary \ref{cor:TII-TII} hold. 
\end{proof}

\begin{example}
Let $n=120$, $q=6$ and $s=115$. Then, $d=20$, $z=5$ and $n=qd$. By Definition \ref{def:TIIpartition class}, any partition class containing $20$ non-negative integers that sum to $5$ is a partition class for some matrix in $\mc M_n^0(6,115)$.

Let $A \in \mc M_n^0(6,115)$ and ${\bf z}=\npmatrix{z_1 & \ldots & z_{20}}\in \PII (A)$. Let $Z'_A$ be a $\hat d \times c$ matrix satisfying $\vecm(Z'_A)^T={\bf z}$:
$$Z'_A=\left(
\begin{array}{cccc}
 z_1 & z_6 & z_{11} & z_{16} \\
 z_2 & z_7 & z_{12} & z_{17} \\
 z_3 & z_8 & z_{13} & z_{18}\\
 z_4 & z_9 & z_{14} & z_{19}\\
 z_5 & z_{10} & z_{15} & z_{20}
\end{array}
\right).$$
We want to determine when $A=B^c$ is a power of some matrix $B \in \mc M_n^0(24,115)$.

From the first part of the corollary, we know that each row in $Z'_A$ must have elements from the set $\{w_i,w_i+1\}$. That is, if any row of $Z'_A$ either contains three different numbers or numbers that are more than one apart, we know that $A$ is not a power of any matrix $B \in \mc M_n^0(24,115)$. The necessary condition to have entries at most one apart in each row of $Z'_A$, and the fact that all elements of $Z'_A$ sum up to $5$, imply that the maximal possible entry in any one row of $Z'_A$ is $2$, and if $2$ is an element of $Z'_A$, then the row that contains it is the only nonzero row in the matrix. 
We consider two cases:
\begin{enumerate}
\item {\bf{$2$ is an element of $Z'_A$.}} If $2$ is an entry in a row then the other entries in that row have to be $1$. Without loss of generality, we can put the entries $2$ and $1$ in the first row of $Z'_A$. This gives us
$$Z=\left(
\begin{array}{cccc}
 1 & 1 & 1 & 2 \\
 0 & 0 & 0 & 0 \\
 0 & 0 & 0 & 0\\
 0 & 0 & 0 & 0\\
 0 & 0 & 0 & 0
\end{array}
\right),$$
and a unique partition class $\PII(A)$. Referring back to Example \ref{Ex:T2-T2}, we see that  $\PII(A)$ is equivalent to $\PII(B^4)$ for $B \in \mc M_n^0(24,115)$ with associated partition class $\T({\bf \hat z})$, where $\hat {\bf z}=\npmatrix{5 & 0 & 0 & 0 & 0}$. In order words, $A=B^4$ for some matrix
$B \in \mc M_n^0(24,115)$ with $\npmatrix{5 & 0 & 0 & 0 & 0} \in \PII(B)$. 

\item {\bf{All elements of $Z'_A$ are either $0$ or $1$.}} Under this constraint
the first item in Corollary \ref{cor:TII-TII} automatically holds, and we have several options for the matrix $Z$ as defined in Remark \ref{remark:TII-TII}. In other words, we can choose the row sums of $Z$ arbitrarily, as long as the sum of all the entries in $Z$ is equal to $5$. Let us look at a few specific examples:
\begin{itemize}
\item Letting $$Z=\left(
\begin{array}{cccc}
 z(1,1) & z(1,2) & z(1,3) & z(1,4) \\
 z(2,1) & z(2,2) & z(2,3) & z(2,4) \\
 z(3,1) & z(3,2) & z(3,3) & z(3,4)\\
 z(4,1) & z(4,2) & z(4,3) & z(4,4)\\
 z(5,1) & z(5,2) & z(5,3) & z(5,4)
\end{array}
\right)=\left(
\begin{array}{cccc}
1  & 1 & 1 & 1 \\
0  & 0 & 0 & 1 \\
0 & 0 & 0 & 0\\
0  & 0 & 0 & 0\\
0  & 0 & 0 & 0
\end{array}
\right),$$ 
we get $w_1=w_2=0$, $w_3=w_4=w_5=-1$ and $\beta_1=0$, $\beta_2=3$, $\beta_3=\beta_4=\beta_5=0$.

In this case $Z'=Z$, and any matrix $A \in \mc M_n^0(6,115)$ with $\PII(A)=\T(\vecm({Z'})^T)$ is a fourth power of some a matrix $B \in \mc M_n^0(24,115)$. 

\item For $$Z=\left(
\begin{array}{cccc}
 0 & 0 & 1 & 1 \\
 0 & 0 & 1 & 1 \\
 0 & 0 & 0 & 0\\
 0 & 0 & 0 & 1\\
 0 & 0 & 0 & 0
\end{array}
\right),$$
we get $w_1=w_2=0$, $\beta_1=\beta_2=2$, $w_3=w_5=-1$, $\beta_3=\beta_5=0$, $w_4=0$ and $\beta_4=3$. We can cyclically permute the rows of $Z$ in many ways, but this $Z$ will give us only one partition class as in Corollary \ref{cor:TII-TII}.

Specifically, $A=B^4$ for some $B \in \mc M_n^0(24,115)$ if and only if $\vecm(Z')^T \in ~ \PII(A)$, where:
$$Z'=\left(
\begin{array}{cccc}
 z(1,1) & z(1,2) & z(1,3) & z(1,4) \\
 z(2,3) & z(2,4) & z(2,1) & z(2,2) \\
 z(3,1) & z(3,2) & z(3,3) & z(3,4)\\
 z(4,1) & z(4,2) & z(4,3) & z(4,4)\\
 z(5,2) & z(5,3) & z(5,4) & z(5,1)
\end{array}
\right)=\left(
\begin{array}{cccc}
 0 & 0 & 1 & 1 \\
 1 & 1 & 0 & 0 \\
 0 & 0 & 0 & 0\\
 0 & 0 & 0 & 1\\
 0 & 0 & 0 & 0
\end{array}
\right).$$
In that case $\PII(B)=\T(\hat {\bf z})$, where
$\hat {\bf z}=\npmatrix{2 & 2 & 0 & 1 & 0}$.
\end{itemize}

\end{enumerate}
\end{example}

\subsection{Type III arc is a power of  a Type III arc}
Throughout this subsection we assume $n=qc\hat d+y$, $d=c\hat d$, $\gcd(qc,y)=1$. By Theorem \ref{thm:powers}: $$\mc K_n(q,qc\hat d+y)=\mc K_n(qc,qc\hat d+y)^c.$$
The theorem below determines the partition class of $B^c$ for $B \in \mc M_n^0(qc,qc\hat d+y)$.

\begin{theorem}\label{thm:TIII-TIII}
Let $B \in \mc M_{n}^0(qc,qc\hat d+y)$ and $\npmatrix{\hat y_1 & \ldots & \hat y_{\hat d}} \in \PIII(B)$. For $i=1\ldots, \hat d$, let $u_i$, $\beta_i$, $\eta_i$ and $\gamma_i$ be defined by:

$$\hat y_i= u_ic+\beta_i \text{ and }
\sum_{k=1}^i\beta_k=\eta_i c+\gamma_i
$$
where $\beta_i, \gamma_i \in \{0,\ldots,c-1\}$. For $i=1,\ldots,\hat d$ and $j=1,\ldots,c$, we define:
\begin{align}\label{eq:T3parts}
y(i,j)=\begin{cases}
u_i+1, &j=\langle\gamma_{i-1}+1\rangle_c,\ldots, \langle\gamma_{i-1}+\beta_i\rangle_c \\
u_i, &j=\langle\gamma_{i-1}+\beta_i+1\rangle_c, \ldots, \langle\gamma_{i-1}+c\rangle_c .
\end{cases}
\end{align}

The partition class $\PIII(B^c)$ is defined as follows:
\begin{itemize}
\item the parts of the partition are equal to $y(i,j)$, $i=1,\ldots,\hat d$, $j=1\ldots,c$, 
\item in the partition $y(i,j)$ is followed by $y(i+1,j)$ for $i=1,\ldots, \hat d-1$. Further, $y(\hat d,j)$ is followed by $y(1,\langle j-y \rangle_c)$.
 \end{itemize}
\end{theorem}

\begin{proof}
Let $B \in \mc M_n^0(\hat q, \hat s)$, $\hat q=qc$, $\hat s=qc\hat d+y$, and $\npmatrix{\hat y_1 & \ldots & \hat y_{\hat d}} \in \PIII(B)$.

\quad

\noindent{\bf Digraph of $B^c$. } Let $\widehat \Gamma$ be a digraph isomorphic to $\Gamma(B)$. In this first step, we will find a directed graph $\widehat \Gamma^{(c)}$ that is isomorphic to $\Gamma(B^c)$.
By Theorem \ref{thm:TIII_realisation},  $\Gamma(B)$ is isomorphic to: 
    $$\widehat \Gamma=C({\bf a}(n)) +\{\hat e_1,\ldots,\hat e_{\hat d}\},$$
where $\hat e_i:=\{( i\hat q+\sum_{k=1}^{i} \hat y_k, 1+ (i-1) \hat q+\sum_{k=1}^{i}\hat y_k)\}$.

Since $\gcd(n,c)=1$,  $C({\bf a}(n))^{(c)}$ is an $n$-cycle by Remark \ref{remark:powerofcycle}: 
\begin{align}\label{eq:T3-n-cycle} 
 C({\bf a}(n))^{(c)}=C(\langle c\cdot {\bf a}(n)\rangle_n).
\end{align}
In addition, each edge  $\hat e_i$ contributes the following $c$ edges to $\widehat \Gamma^{(c)}$:
\begin{align}\label{eq:T3-connecting edges}
e_{i,t}:=(\langle 1+ i \hat q+\sum_{k=1}^i \hat y_k-t\rangle _n, \langle 1+(i-1)\hat q+\sum_{k=1}^i \hat y_k+c-t\rangle_n), \, t=1,\ldots,c.
\end{align}
 We denote $e_{i,t}:=(v_O(i,t),v_I(i,t))$ for future use.
We have:
$$\widehat \Gamma^{(c)}=C(\langle c\cdot {\bf a}(n)\rangle_n)+\{e_{i,t}:i=1,\ldots,\hat d,t=1,\ldots c\}.$$
Note that $\widehat \Gamma^{(c)}$ consists of an $n$-cycle $C(\langle c\cdot {\bf a}(n)\rangle_n)$ together with $\hat d c$ $q$-cycles, where each $q$-cycle is formed by an edge $e_{i,t}$ connecting two vertices of the $n$-cycle. Also, the weight $1-\alpha$ on the edges $\hat e_i,i=1,\ldots,\hat d$, in $\widehat \Gamma$ gives the weight $1-\alpha$ to the edges $e_{i,t},i=1,\ldots,\hat d,t=1,\ldots,c$, in $\widehat \Gamma^{(c)}$. Thus, the weights on all the $q$-cycles, $C_{i,j},i=1,\ldots,\hat d, j=1,\ldots,c$, are equal to $1-\alpha$ in in $\widehat \Gamma^{(c)}$.

\quad

\noindent{\bf Congruence modulo $c$.} 
To determine the partition class of $B^c$, we will study separately each part of the graph $\widehat \Gamma^{(c)}$ that (for a fixed $j$) involves vertices congruent to $j$ modulo $c$. Let $\Pi_j:=P(j+c {\bf a}_0(h(j)))$, $j=1,\ldots,c$,  where $$ h(j):=
\begin{cases}
\lfloor \frac{n}{c} \rfloor +1&\text{ for }j =1,\ldots,\langle n \rangle_c\\
\lfloor \frac{n}{c} \rfloor &\text{ for }j =\langle n \rangle_c+1,\ldots,c,
\end{cases}$$
or equivalently:
$$ h(j)=1+\hat d q+\left\lfloor\frac{y-j}{c}\right\rfloor.$$
With this notation, we have:
$$C({\bf a}(n))^{(c)}=\cup_{j=1}^c \left(\Pi_j + \{(j+y+(q\hat d -1)c,j)\}\right).$$
From  $j+y+(q\hat d-1)c \in V(\Pi_{\langle j+y\rangle_c})$ and $j \in V(\Pi_j)$, we deduce that $\Pi_{\langle j+y\rangle_c}$
is connected to $\Pi_j$ with the edge $(j+y+(q\hat d -1)c,j)$.
In particular,  $\Pi_j$ is followed by $\Pi_{\langle j-y \rangle_c}$ in $\widehat \Gamma^{(c)}$.

From equation \eqref{eq:T3-connecting edges} we notice that $\langle v_O(i,t) \rangle_c=\langle v_I(i,t) \rangle_c$. This implies that for any pair $i\in\{1,\ldots,\hat d\}$ and $t\in\{1,\ldots,c\}$,
there exists a unique $j\in \{1,\ldots, c\}$ such that the edge $e_{i,t}$ connects two  vertices in $\Pi_j$.
Moreover,  the vertices of $e_{i,t}$ belong to $\Pi_j$ precisely when $t$ is congruent to $ (1-j+\sum_{k=1}^i \hat y_k) \mod c$.
We define $e'_{i,j}:=e_{i,\langle 1-j+\sum_{k=1}^i \hat y_k\rangle_c}$ with $e'_{i,j}=(v'_O(i,j),v'_I(i,j))$, and note that
$$\{e_{i,t}:i=1,\ldots,\hat d,t=1,\ldots c\}=\{e_{i,j}':i=1,\ldots,\hat d,j=1,\ldots c\}.$$

From now on we will work with edges $e_{i,j}'$. We write: $v'_O(i,j)=j+k(i,j)c$ and $v'_I(i,j)=j+l(i,j)c$, where:
\begin{align}\label{eq:kij}
k(i,j)&=iq+\left\lfloor\frac{\sum_{k=1}^i \hat y_k -j}{c}\right\rfloor , 
\end{align}

\begin{align}\label{eq:lij}
l(i,j)&=1+(i-1)q+\left\lfloor\frac{\sum_{k=1}^i \hat y_k -j}{c}\right\rfloor.
\end{align}

\noindent{\bf Parts of the partition. }
We fix $j$, and consider those $q$-cycles in $\widehat \Gamma^{(c)}$ whose  vertices are contained in $V(\Pi_j)=\{j,j+c,\ldots,j+(h(j)-1) c\}$. 
From above we already know that those are precisely the $q$-cycles in $\widehat \Gamma^{(c)}$ that contain an edge $e_{i,j}'$ for some $i=1,\ldots,\hat d$. 

The first vertex in $\Pi_j$ is $j+0c$, followed by vertices $j+c$, $j+2\cdot c$,\ldots, $j+(h(j)-1)\cdot c$, in this order. From  $k(1,j)<k(2,j)<\ldots <k(c,j)$ we deduce that the first $q$-cycle on $V(\Pi_j)$ will contain $e_{1,j}'$, followed by the $q$-cycle containing $e_{2,j}'$, etc. 
The last $q$-cycle in $\Pi_j$, made by the edge $e'_{\hat d,j}$, contains the vertex $v'_O(\hat d,j)=j+k(\hat d,j)c$. Since $k(\hat d,j)=h(j)-1$, we conclude that $v'_O(\hat d,j)$ is the last vertex in $\Pi_j$. In other words, there are no vertices in $\Pi_j$ after the last $q$-cycle.

We define $y(1,j):=l(1,j)-0$ to be the number of vertices that lie before $v'_I(1,j)$ in $\Pi_j$. Next, we focus on the path between two neighbouring $q$-cycles inside our fixed $\Pi_j$. More specifically, for $i=2,\ldots,\hat d$, we want to determine
the number of vertices that are not contained in any $q$-cycle and are on the path connecting the $q-$cycles made by the edges $e'_{ i-1, j}$ and $e'_{i, j}$. The first vertex on this path is $v'_O( i-1 ,j)=j+k(i-1,j)c$ and the last vertex is $v'_I(i, j)=j+l(i,j)c$. The number of vertices on $\Pi_j$ (strictly) between $v'_O( i-1 ,j)$ and $v'_I(i, j)$ is equal to:
$y(i,j):=l(i,j)-k(i-1,j)-1.$ 

In summary, the vector ${\bf y}(j):=\npmatrix{y(1,j) & \ldots & y(\hat d, j)}$ contains the contribution to partitions in the partition class of $\PIII(B^c)$ coming from the part of the graph involving $V(\Pi_j)$.

To determine $y(i,j)$ we write $\hat y_i=u_i c+\beta_i$, where $\beta_i \in \{0,\ldots,c-1\}$, and $\sum_{k=1}^i \beta_k=\eta_i c+\gamma_i$, where  $\gamma_i \in \{0,\ldots,c-1\}$. Inserting $i=1$ in \eqref{eq:lij} we now get:
\begin{align*}
   y( 1,j)&= 1+\left\lfloor\frac{\hat y_1 -j}{c}\right\rfloor =
\begin{cases}
       u_1+1, &\text{ for }j \in \{1,\ldots,\beta_1\}\\
u_1 &\text{ for }j\in \{\beta_1+1,\ldots,c\}.
\end{cases}
\end{align*}
To determine $y(i,j)$ from \eqref{eq:kij} and \eqref{eq:lij} we compute:
\begin{align*}
y(i,j)&=l(i,j)-k(i-1,j)-1 \\
&=u_i+\left\lfloor\frac{\sum_{k=1}^i\beta_k -j}{c}\right\rfloor-\left\lfloor\frac{\sum_{k=1}^{i-1}\beta_k -j}{c}\right\rfloor\\
&=u_i+\left \lfloor \frac{\gamma_{i-1}+\beta_i-j}{c}\right \rfloor-\left \lfloor \frac{\gamma_{i-1}-j}{c}\right \rfloor . 
\end{align*}
We distinguish between two cases: $\gamma_{i-1}+\beta_{i}<c$, and  $\gamma_{i-1}+\beta_{i}\geq c$. If  $\gamma_{i-1}+\beta_{i}<c$, then $\eta_i=\eta_{i-1}$,
$\gamma_i=\gamma_{i-1}+\beta_i$, and 
$$y(i,j)=\begin{cases}
u_i, &j=1,\ldots, \gamma_{i-1} \\
u_i+1, &j=\gamma_{i-1}+1,\ldots, \gamma_{i-1}+\beta_i \\
u_i, &j=\gamma_{i-1}+\beta_i+1,\ldots, c .\\
\end{cases}$$
For $\gamma_{i-1}+\beta_{i}\geq c$ we get
$\eta_i=\eta_{i-1}+1$, $\gamma_i=\gamma_{i-1}+\beta_i-c$, and 
$$y(i,j)=\begin{cases}
u_i+1, &j=1,\ldots, \gamma_{i-1}+\beta_i-c \\
u_i, &j=\gamma_{i-1}+\beta_i-c+1,\ldots, \gamma_{i-1}\\
u_i+1, &j=\gamma_{i-1}+1,\ldots, c .\\
\end{cases}$$
Both cases can be written in one expression as given by \eqref{eq:T3parts}.

\quad

\noindent{\bf Final ordering.}
So far we have proved that ${\bf y} \in \PIII(B^c)$ consists of ${\bf y}(j)$, $j=1,\ldots,c$, in some order. Since we also know that   $\Pi_j$ is followed by $\Pi_{\langle j-y \rangle_c}$ in $\widehat \Gamma^{(c)}$, we conclude that ${\bf y}(j)$ is followed by ${\bf y}({\langle j-y \rangle_c})$ in ${\bf y} \in \PIII(B^c)$. With this final observation, the partition class is uniquely defined.  
\end{proof}

\begin{remark}\label{remark:TIII-TIII}
Given $\hat{\mathbf{y}} \in \PIII(B)$ we can define a matrix $Y'$ that satisfies $\vecm(Y')^T \in \PIII(B^c)$ as follows. From $\hat{\mathbf{y}}$ we define $y(i,j)$ as in the statement of Theorem \ref{thm:TIII-TIII}, and form the matrix:
$$Y:=\left(
\begin{array}{cccc}
 y(1,1) & y(1,2) & \ldots & y(1,c) \\
 \vdots & \vdots & \ddots & \vdots \\
 y(\hat d,1) & y(\hat d,2) & \ldots & y(\hat d, c)\\
\end{array}
\right).$$
The matrix $Y'$ is obtained from $Y$ by a permutation of columns such that the $j$-th column of $Y'$ is equal to ${\bf y}(\langle 1-(j-1)y \rangle_c)$, which is the $\langle 1-(j-1)y \rangle_c$-th column of $Y$. The partition class $\PIII(B^c)$ is then equal to $\T(\bf y)$, where ${\bf y}=\vecm(Y')^T$.  
\end{remark}

Let $n=337$ and consider the arc $\mc K_n(q,s)=\mc K_n(27,337)$. Using Theorem \ref{thm:powers} we have:
\begin{align}\label{eq:T3arcpower}
  \mc K_n(27,337)=\mc K_n(337,324)^{12}=\mc K_n(81,337)^3=\mc K_n(108,337)^4  . 
\end{align}
We will consider the case  $\mc K_n(27,337)=\mc K_n(337,324)^{12}$ after Theorem \ref{thm:TI-TIII}. In the example below we illustrate the proof of Theorem \ref{thm:TIII-TIII} for $\mc K_n(27,337)=\mc K_n(81,337)^3$.

\begin{example}\label{ex:T3-T3}
Let $n=337$, then $\mc K_n(27,337)=\mc K_n(81,337)^3$, where $q=27$, $d=12$, $s=337$, $c=3$, $\hat q=81$, $\hat d=4$, $y=13$ and $y=\hat y_1 + \hat y_2+ \hat y_3+ \hat y_4$.
Let $B\in \mc M_n^0(81,337)$. We will illustrate the parts of the proof of Theorem \ref{thm:TIII-TIII} for a few possible choices for $\hat{\mathbf{y}} \in \PIII(B)$.
\begin{enumerate}
\item Let $\npmatrix{\hat y_1 & \hat y_2 & \hat y_3 & \hat y_4 }=\npmatrix{5 & 3 & 3 & 2 }.$
\quad

\noindent{\bf Digraph of $B^c$. }$\Gamma(B)$ is isomorphic to the following digraph $\widehat \Gamma$,  
$$\widehat \Gamma=C({\bf a}(337)) + \{\hat e_i:i=1,\ldots,4\},$$ 

where $\hat e_i=\{(81i+\sum_{k=1}^{i} \hat y_k, 1+ 81(i-1) +\sum_{k=1}^{i}\hat y_k)\}.$

Note that $\widehat \Gamma$ consists of a $337$-cycle $C({\bf a}(337))$ and four cycles of order $81$ that are made by the edges $\hat e_i$. The $337$-cycle and the edges $\hat e_i$ of $\widehat \Gamma$ give another $337$-cycle and the edges $e_{i,t}$, respectively, in $\widehat \Gamma^{(3)}$. These are explained below in detail.

By Theorem \ref{thm:TIII-TIII}, $\Gamma(B^3)$ is isomorphic to $\widehat \Gamma^{(3)}$, given by:

$$\widehat \Gamma^{(3)}=C(\langle 3\cdot {\bf a}(337)\rangle_n)+\{e_{i,t}:i=1,\ldots,4,t=1,\ldots 3\},$$

where $e_{i,t}:=(\langle 81 i +\sum_{k=1}^i \hat y_k-t+1\rangle _n, \langle 81(i-1)+\sum_{k=1}^i \hat y_k+3-t\rangle_n).$ Table \ref{table:edges-e(i,t)} shows the edges $e_{i,t}$. 
\begin{table}[!h]
\begin{center}
\begin{tabular}{| c |c |c|c|}
\hline
 & $t=1$ & $t=2$ & $t=3$\\
\hline
$e_{1,t}$ & $(86,8)$ & $(85,7)$ & $(84,6)$ \\
\hline
$e_{2,t}$ & $(170,92)$ & $(169,91)$ & $(168,90)$ \\
\hline
$e_{3,t}$ & $(254,176)$ & $(253,175)$ & $(252,174)$ \\
\hline
$e_{4,t}$ & $(337,259)$ & $(336,258)$ & $(335,257)$  \\ 
\hline
\end{tabular}
\caption{\label{table:edges-e(i,t)} The edges $e_{i,t}$}
\end{center}
\end{table}

\noindent{\bf Congruence modulo $c$.}
The $n$-cycle of $\widehat \Gamma^{(3)}$ can be written in terms of the paths $\Pi_j,j=1,2,3$, as follows:
 $$C(\langle 3\cdot {\bf a}(337)\rangle_n)=\cup_{j=1}^3 \left(\Pi_j + \{(j+334,j)\}\right),$$

 where $\Pi_j=P(j+3 {\bf a}_0(h(j)))$ and $h(j)=109+\left\lfloor\frac{13-j}{3}\right\rfloor$. Also, recall that the vertices of the edge $e'_{i,j}$ belong $\Pi_j$ where, $e'_{i,j}=e_{i, \langle 1-j+\sum_{k=1}^i \hat y_k\rangle_c}$. 

 For example, $j=1$ gives us the path $\Pi_1=P(1+3{\bf a}_0(113))$, and the edges $e_{i,1}'$, $i=1,2,3,4$, connect the vertices in $\Pi_1$ to form four cycles of order $27$. 

Writing $e'_{i,1}=(1+3k(i,1), 1+3l(i,1))$, we get the values given in Table \ref{table:k(i,j)-l(i,j)} for $k(i,1)$ and $l(i,1)$. 
\begin{table}[!h]
\begin{center}
\begin{tabular}{| c |c |c|c|}
\hline
 & $e'_{i,1}$ & $k(i,1)$ & $l(i,1)$\\
\hline
$i=1$ & $(85,7)$ & $28$ & $2$ \\
\hline
$i=2$ & $(169,91)$ & $56$ & $30$ \\
\hline
$i=3$ & $(253,175)$ & $84$ & $58$ \\
\hline
$i=4$ & $(337,259)$ & $112$ & $86$  \\ 
\hline
\end{tabular}
\caption{\label{table:k(i,j)-l(i,j)}$k(i,1)$ and $l(i,1)$ for $\Pi_1$}
\end{center}
\end{table}

\noindent{\bf Parts of the partition. } 
To illustrate the next part of the proof, let us continue with $j=1$. Considering the vertices of $\Pi_1$ in the natural path order i.e. $V(\Pi_1)=\{1,1+3,\ldots,1+3(113-1)\}$, we see that the first vertex in $\Pi_1$ contained in a $27$-cycle (made by the edge $e'_{1,1}$) is the vertex, $v'_I(1,1)=7=1+3(2)$. This implies, $l(1,1)=2$, and $v'_I(1,1)$ is the third vertex of $\Pi_1$. Therefore, $y(1,1)=2$ (recalling that $y(i,j)$ is defined to be the number of vertices in $\Pi_j$ before the vertex $v'_I(i,j)$).

Next, $y(2,1)$ is the number of vertices not contained in a $27$-cycle strictly between $v'_O(1,1)$ and $v'_I(2,1)$. Since, 
$v'_O(1,1)=85=1+3(28)$ and $v'_I(2,1)=91=1+3(30)$, we have  $k(1,1)=28$ and $l(2,1)=30$. This implies, $y(2,1)=l(2,1)-k(1,1)-1=1.$ Similarly, $y(3,1)=l(3,1)-k(2,1)-1=1$ and $y(4,1)=l(4,1)-k(3,1)-1=1.$ Also, note that $v'_O(4,1)=337=1+3(112)$ is the last  vertex of $\Pi_1$.

More generally, to get all the parts of the partition we express the $k(i,j)$ and $l(i,j)$ in terms of the parameters $u_i$, $\beta_i$, and $\gamma_i$. In this case, we write $ \hat y_i=3u_i+\beta_i$ and 
$\sum_{k=1}^i \beta_k=3\eta_i+\gamma_i$,
$\beta_i, \gamma_i\in \{0,1,2\}$, to determine the following parameters: 
\begin{align*}
u_1&=u_2=u_3=1, \, u_4=0, \\ \beta_1&=\beta_4=2, \, \beta_2=\beta_3=0, \\
\gamma_1&=\gamma_2=\gamma_3=2, \, \gamma_4=1.
\end{align*}

 Using equation \eqref{eq:T3parts}, we get $y(i,j)$, $i=1,\ldots,4$, $j=1,\ldots,3$. Thus, the $Y$ matrix (as defined in Remark \ref{remark:TIII-TIII}) is equal to:
$$Y=\left(
\begin{array}{ccc}
 y(1,1) &  y(1,2) & y(1,3) \\
 y(2,1) &  y(2,2) & y(2,3) \\
 y(3,1) &  y(3,2) & y(3,3)\\
 y(4,1) &  y(4,2) & y(4,3)\\
 \end{array}
\right)=\left(
\begin{array}{ccc}
 2 & 2 & 1  \\
 1 & 1 & 1  \\
 1 & 1 & 1 \\
 1 & 0 & 1 \\
\end{array}
\right).$$

In particular, the elements of $\PIII(B^3)$ have one part equal to $2$, one part equal to $0$ and all other parts equal to $1$.

\noindent{\bf Final ordering.} To determine $\PIII(B^3)$, we still need the ordering of the parts. For this we look at the $Y'$ matrix in Remark \ref{remark:TIII-TIII}:

$$Y'=\left(
\begin{array}{ccc}
 y(1,1) &  y(1,3) & y(1,2) \\
 y(2,1) &  y(2,3) & y(2,2) \\
 y(3,1) &  y(3,3) & y(3,2)\\
 y(4,1) &  y(4,3) & y(4,2)\\
\end{array}
\right)=\left(
\begin{array}{ccc}
 2 & 1 & 2  \\
 1 & 1 & 1  \\
 1 & 1 & 1 \\
 1 & 1 & 0 \\
\end{array}
\right).$$

Thus, the partition class $\PIII(B^3)$ is given by ${\bf y}=\T(\vecm(Y')^T).$

\item Let $\npmatrix{\hat y_1 & \hat y_2 & \hat y_3 & \hat y_4 }=\npmatrix{11 & 2 & 0 & 0}.$

Writing $ \hat y_i=3u_i+\beta_i$ and 
$\sum_{k=1}^i \beta_k=3\eta_i+\gamma_i$
$\beta_i, \gamma_i\in \{0,1,2\}$, gives:
$u_1=3 $, $u_2=u_3=u_4=0$;  $\beta_1=\beta_2=2$, $\beta_3=\beta_4=0$ and $\gamma_1=2$, $\gamma_2=\gamma_3=\gamma_4=1$.
Using equation \eqref{eq:T3parts} to determine $y(i,j)$, $i=1,\ldots,4$, $j=1,\ldots,3$, we get the following matrices (defined in Remark \ref{remark:TIII-TIII}):
$$Y=\left(
\begin{array}{cccc}
 4 & 4 & 3 \\
 1 & 0 & 1 \\
 0 & 0 & 0\\
 0 & 0 & 0 \\
\end{array}
\right) \text{ and } Y'=\left(
\begin{array}{cccc}
 4 & 3 & 4 \\
 1 & 1 & 0 \\
 0 & 0 & 0\\
 0 & 0 & 0 \\
\end{array}
\right).$$ 
The partition class $\PIII(B^3)$ is given by ${\bf y}=\T(\vecm(Y')^T).$

\item Let $\npmatrix{\hat y_1 & \hat y_2 & \hat y_3 & \hat y_4 }=\npmatrix{13 & 0 & 0 & 0}.$ From $ \hat y_i=3u_i+\beta_i$ and 
$\sum_{k=1}^i \beta_k=3\eta_i+\gamma_i$,
$\beta_i, \gamma_i\in \{0,1,2\}$, we get:
$u_1=4 $, $u_2=u_3=u_4=0$;  $\beta_1=1$, $\beta_2=\beta_3=\beta_4=0$ and $\gamma_1=\gamma_2=\gamma_3=\gamma_4=1$.
In this case:
$$Y=Y'=\left(
\begin{array}{cccc}
 5 & 4 & 4 \\
 0 & 0 & 0 \\
 0 & 0 & 0\\
 0 & 0 & 0 \\
\end{array}
\right),$$ 
and the partition class $\PIII(B^3)$ is given by ${\bf y}=\T(\vecm(Y')^T).$

\item Let $\npmatrix{\hat y_1 & \hat y_2 & \hat y_3 & \hat y_4 }=\npmatrix{4 & 3 & 3 & 3}.$ 
Writing $\hat y_i=3u_i+\beta_i$ and $\sum_{k=1}^i \beta_k=~3\eta_i+~\gamma_i$
$\beta_i, \gamma_i\in \{0,1,2\}$ gives:
$u_1=~u_2=u_3=u_4=1$; $\beta_1=1$, $\beta_2=\beta_3=\beta_4=0$ and $\gamma_1=\gamma_2=\gamma_3=~\gamma_4=1$.

Again,
$$Y=Y'=\left(
\begin{array}{cccc}
 2 & 1 & 1 \\
 1 & 1 & 1 \\
 1 & 1 & 1\\
 1 & 1 & 1 \\
\end{array}
\right)$$ 
and the partition class $\PIII(B^3)$ is given by ${\bf y}=\T(\vecm(Y')^T).$
\end{enumerate}
\end{example}

\begin{corollary}\label{cor:TIII-TIII}
Let $A\in \mc M_n^0(q,qc\hat d+y)$ and $\npmatrix{ y_1 & \ldots &  y_{c\hat d}} \in \PIII(A)$. Then $A=B^c$ for some $B \in \mc M_n^0(qc,qc\hat d+y)$ if and only if:
\begin{enumerate}
    \item For every $i\in \{1,\ldots,\hat d\}$ there exists $u_i$ so that $y_{t\hat d+i} \in \{u_i,u_i+1\}$ for $t=0,\ldots,c-1$. (If for some $i$ all $y_{t\hat d+i}$ are equal, then we say that they are equal to $u_i$.) 
    
    We define, $\beta_i:=|\{t\in \{0\ldots,c-1\}; y_{t \hat d+i}=u_i+1\}|$, $i=1,\ldots,\hat d$, and $\sum_{k=1}^i \beta_i:=\eta_i c+\gamma_i$, where $\gamma_i \in \{0,\ldots,c-1\}$. In addition, $\gamma_0:=0$.

    With these parameters we define $y(i,j)$ as follows :
$$y(i,j):=\begin{cases}
u_i+1, &j=\langle\gamma_{i-1}+1\rangle_c,\ldots, \langle\gamma_{i-1}+\beta_i\rangle_c \\
u_i, &j=\langle\gamma_{i-1}+\beta_i+1\rangle_c, \ldots, \langle\gamma_{i-1}+c\rangle_c
\end{cases},$$
where $i=1,\ldots,\hat d$, $j=1\ldots,c$.

    \item The partition class $\PIII(A)$ consists of the parts $y(i,j)$, $i=1,\ldots,\hat d$, $j=1\ldots,c$. Further, the part $y(i,j)$ is followed by $y(i+1,j)$ for $i=1,\ldots, \hat d-1$ and $y(\hat d,j)$ is followed by $y(1,\langle j-y \rangle_c)$ in $\PIII(A)$.
 \end{enumerate}
If the conditions above are satisfied, then $A=B^c$ for $B \in \mc M_n^0(qc,qc\hat d+y)$ with $\PIII(B)=\T(\npmatrix{\hat y_1 &\ldots &\hat y_{\hat d}}),$  where $\hat y_i:=\sum_{j=1}^c y(i,j)$, for $i=1,\ldots,\hat d$.

\end{corollary}
\begin{proof}
The result follows directly from Theorem \ref{thm:TIII-TIII}. 
\end{proof}

\begin{example}\label{ex:corT3-T3}
Let $n=s=337$, $q=27$, $d=12$, $c=3$, $y=13$ and $n=s=qd+y$. By Definition \ref{def:TIIIpartition class}, any partition class containing $12$ non-negative integers that sum to $13$ is a partition class for some matrix in $\mc M_n^0(27,337)$.

Let $A\in \mc M_n^0(27,337)$ and ${\bf y}=\npmatrix{y_1 & \ldots & y_{12}}\in \PIII(A)$. Let $Y'_A$ be a $\hat d \times c$ matrix satisfying $\vecm(Y'_A)^T={\bf y}$:

$$Y'_A=\left(
\begin{array}{ccc}
 y_1 & y_5 & y_9 \\
 y_2 & y_6 & y_{10} \\
 y_3 & y_7 & y_{11} \\
 y_4 & y_8 & y_{12} \\
 \end{array}
\right).$$
We want to determine when $A=B^3$ for some matrix $B \in \mc M_n^0(81,337)$.

From the first part of Corollary \ref{cor:TIII-TIII}, we know that for $A$ to be a power of any matrix $B \in \mc M_n^0(81,337)$, each row in $Y'_A$ must have elements from a set of the form $\{u_i,u_i+1\}$ for some integer $u_i \geq 0$. Under this constraint, we can choose the row sums of $Y'_A$ in many possible ways taking care that the sum of all the entries in $Y'_A$ is equal to $13$. Let us look at a few specific examples.  
\begin{enumerate}
\item Let $$Y'_A=\left(
\begin{array}{cccc}
 1 & 1 & 0 \\
 0 & 0 & 0 \\
 2 & 2 & 2\\
 1 & 2 & 2\\
 \end{array}
\right).$$
From the first item of Corollary \ref{cor:TIII-TIII}, we get:
$u_1=u_2=0$, $u_3=2$, $u_4=1$, $\beta_1=2$, $\beta_2=\beta_3=0$, $\beta_4=2$ and $\gamma_1=\gamma_2=\gamma_3=2$, $\gamma_4=1.$

Thus,
$$Y=\left(
\begin{array}{ccc}
 y(1,1) & y(1,2) & y(1,3)  \\
 y(2,1) & y(2,2) & y(2,3)  \\
 y(3,1) & y(3,2) & y(3,3) \\
 y(4,1) & y(4,2) & y(4,3) \\
\end{array}
\right)=\left(
\begin{array}{cccc}
 1 & 1 & 0 \\
 0 & 0 & 0 \\
 2 & 2 & 2\\
 2 & 1 & 2\\
 \end{array}
\right).$$

Using the second item of Corollary \ref{cor:TIII-TIII}, we get:
$$Y'=\left(
\begin{array}{ccc}
 y(1,1) & y(1,3) & y(1,2)  \\
 y(2,1) & y(2,3) & y(2,2)  \\
 y(3,1) & y(3,3) & y(3,2) \\
 y(4,1) & y(4,3) & y(4,2) \\
 \end{array}
\right)=\left(
\begin{array}{cccc}
 1 & 0 & 1 \\
 0 & 0 & 0 \\
 2 & 2 & 2\\
 2 & 2 & 1\\
 \end{array}
\right).$$

We see that $\vecm(Y'_A)^T$ is permutationally equivalent to $\vecm(Y')^T$ and thus satisfies all the conditions of Corollary \ref{cor:TIII-TIII}. Therefore, $A=B^3$ for $B \in \mc M_n^0(81,337)$ with $\PIII(B)=\T(\hat {\bf y})$, where $\hat {\bf y}=\npmatrix{6 & 5 & 2 & 0}$.
\item For $$Y'_A=\left(
\begin{array}{cccc}
 4 & 4 & 3 \\
 1 & 0 & 1 \\
 0 & 0 & 0\\
 0 & 0 & 0\\
 \end{array}
\right),$$
we get $u_1=3$, $u_2=u_3=u_4=0$; $\beta_1=\beta_2=2$, $\beta_3=\beta_4=0$ and $\gamma_1=2$,$\gamma_2=\gamma_3=\gamma_4=1$. 

Thus,
$$Y=\left(
\begin{array}{cccc}
 4 & 4 & 3 \\
1 & 0 & 1 \\
 0 & 0 & 0\\
 0 & 0 & 0\\
 \end{array}
\right).$$
Referring back to Example \ref{ex:T3-T3} (Item 2.) we know that $$Y'=\left(
\begin{array}{cccc}
 4 & 3 & 4 \\
1 & 1 & 0 \\
 0 & 0 & 0\\
 0 & 0 & 0\\
 \end{array}
\right).$$
This time the second item of Corollary \ref{cor:TIII-TIII} is not satisfied. Hence, $A \ne B^3$ for any $B \in \mc M_n^0(81,337)$. 

\item Let $$Y'_A=\left(
\begin{array}{cccc}
 2 & 1 & 1 \\
 1 & 1 & 1 \\
 1 & 1 & 1\\
 1 & 1 & 1\\
 \end{array}
\right).$$
In this case we get: $u_1=u_2=u_3=u_4=1$; $\beta_1=1$, $\beta_2=\beta_3=\beta_4=0$ and $\gamma_1=\gamma_2=\gamma_3=\gamma_4=1$.
Thus, $$Y=\left(
\begin{array}{cccc}
 2 & 1 & 1 \\
 1 & 1 & 1 \\
 1 & 1 & 1\\
 1 & 1 & 1\\
 \end{array}
\right).$$

From Item $4$ of Example \ref{ex:T3-T3} we note that $Y'=Y$ and $A=B^3$ for $B \in \mc M_n^0(81,337)$ with the $\PIII(B)=\T(\hat {\bf y})$, where $\hat {\bf y}=\npmatrix{4 & 3 & 3 & 3}.$

\end{enumerate}
\end{example}

\subsection{Type III arc is a power of a Type I arc} In this subsection we assume
$n=qd+y$, $\gcd(qd,y)=1$, $y \in \{1,\ldots, q-1\}$. By Theorem \ref{thm:powers}: $$\mc K_n(qd+y,qd)^d=\mc K_n(q,qd+y).$$ 
The following theorem determines the partition class of $B^d$ for $B \in \mc M_n^0(qd+y,qd)$. We note that the result can also be obtained from Theorem \ref{thm:TIII-TIII} by taking $i=1$, $\hat d=1$, and $c=d$.

\begin{theorem}\label{thm:TI-TIII}
Let $B \in \mc M_{n}^0(qd+y,qd)$, where
$n=s=qd+y$, $\gcd(d,s)=1$ and $y \in \{1, \ldots, q-1\}$. We define $u$ and $\beta$ by writing $y:= u d+\beta$, where $\beta\in \{0,\ldots,d-1\}$. For $j=1,\ldots,d$ we define:
\begin{align}\label{eq:T1-T3parts}
    y(j):=\begin{cases}
u+1, & j=1,\ldots,\beta\\
u, & j=\beta+1,\ldots,d . 
\end{cases}
\end{align}
The partition class $\PIII(B^d)$ is defined as follows:
\begin{itemize}
\item the parts of the partition are equal to $y(j)$, $j=1,\ldots,d$, 
\item in the partition $y(j)$ is followed by $y(\langle j-y \rangle_d)$.
\end{itemize}
\end{theorem}

\begin{proof}
Let $B \in M_{n}^0(\hat q, \hat s)$, where $\hat q=qd$, $\hat s=qd+y$, $\gcd(\hat q,n)=1$. 
By Theorem \ref{thm:TI_realisation}, $\Gamma(B)$ is isomorphic to $C({\bf a}(n)) +\{(qd,1)\}$, which is isomorphic to: 
    $$\widehat \Gamma:=C({\bf a}(n)) +\{(\hat s,y+1)\}.$$

\quad

\noindent{\bf Digraph of $B^d$. } The digraph of $\Gamma(B^d)$ is isomorphic to $\widehat \Gamma^{(d)}$:
$$\widehat \Gamma^{(d)}=C(\langle d\cdot {\bf a}(n)\rangle_n)+\{e_{t}:t=1,\ldots d\},$$ where $$e_t=(s-t+1,y+d-t+1), t=1,\ldots d.$$
In particular, $\widehat \Gamma^{(d)}$ consists of an $n$-cycle $C(\langle d\cdot {\bf a}(n)\rangle_n)$ together with $d$ $q$-cycles, where each $q$-cycle is formed by an edge $e_{t}$ connecting two vertices of the $n$-cycle. The edge $(\hat s,y+1)$ in $\widehat \Gamma$ and the edges $e_t,t=1,\ldots,d$, in $\widehat \Gamma^{(d)}$ have weights equal to $1-\alpha$. Equivalently, all the $q$-cycles in $\widehat \Gamma^{(d)}$ have the weight $1-\alpha$.

\quad

\noindent{\bf Congruence modulo $d$.} Let
$\Pi_j:=P(j+d {\bf a}_0(h(j)))$, where $ h(j)=1+ q+\left\lfloor\frac{y-j}{d}\right\rfloor$.
We have:
$$C({\bf a}(n))^{(d)}=\cup_{j=1}^d \left(\Pi_j + \{((q -1)d+y+j,j)\}\right),$$
and note that $\Pi_j$ is followed by $\Pi_{\langle j-y \rangle_d}$ (equivalently, $\Pi_{\langle j+y \rangle_d}$ is followed by $\Pi_j$) in $\widehat \Gamma^{(d)}$. Both vertices of $e_t$ belong to $\Pi_j$, where $t$ is congruent to $\langle 1+y-j \rangle_d$. In other words, the edge $e_{\langle 1+y-j \rangle_d}$ connects two vertices in $\Pi_j$ to form a $q$-cycle in $\widehat \Gamma^{(d)}$. We define $e_j':=e_{\langle 1+y-j \rangle_d}$ and note that $\{e_t:t=1,\ldots,d\}=\{e_j':j=1,\ldots,d\}$.

Let $e_j'=(v_O'(j),v_I'(j))$ and $y=u d+\beta$, $\beta \in \{0,\ldots,d-1\}$. Then, $v_O'(j)=j+k(j)d$ and $v_I'(j)=j+l(j)d$, where 

\begin{align}
k(j):=
    \begin{cases}
q+u &\text{ for }j \leq \beta\\
q+u-1 &\text{ for }j >\beta , 
\end{cases}
\end{align}

\begin{align}
l(j):=
    \begin{cases}
u+1 &\text{ for }j \leq \beta\\
u &\text{ for }j >\beta .
\end{cases}
\end{align}

\quad

\noindent{\bf Parts of the Partition.}  For a fixed $j$, there is only one $q$-cycle made by the edge $e'_j=(v'_O(j),v'_I(j))$ in $\Pi_j$. Also, $k(j)=h(j)-1$ implies that $v'_O(j)=j+k(j)d$ is the last vertex in $\Pi_j$.

We define $y(j):=l(j)-0$ as the number of vertices that lie before $v'_I(j)$ in $\Pi_j$. Thus, $y(j)$ is equal to $l(j)$ and is as defined in \eqref{eq:T1-T3parts} in the statement of the theorem.

\quad

\noindent{\bf Ordering of Parts.} Since the $q$-cycle with the edge $e_j'$ is followed by the $q$-cycle containing the edge $e'_{\langle j-y \rangle_d}$, the part $y(j)$ is followed by the part $y(\langle j-y \rangle_d)$ in the partition. 

\end{proof}

\begin{remark}\label{remark:TI-TIII}
  For $B \in \mc M_n^0(qd,qd+y)$, let us define the row vector ${\bf y}$ consisting of parts of the partition class $\PIII(B^d)$, i.e
 ${\bf y}:=\npmatrix{y(1) & \ldots & y(d)}.$ Since, $y(j)$ is followed by $y(\langle j-y \rangle_{d})$ in the partition class $\PIII(B^d)$, we can permute the elements of ${\bf y}$ to get the row vector:
 $${\bf y'}:=\npmatrix{y(1) & y(\langle 1-y \rangle_d) & \ldots & y(\langle 1-(d-1)y) \rangle_d} \in \PIII(B^d).$$
The partition class $\PIII(B^d)$ is equal to $\T(\bf y')$.  
\end{remark}

\begin{example}
For $n=337$  we have $\mc K_n(27,337)=\mc K_n(337,324)^{12}$, where $q=27$, $n=s=337$, $d=12$,  $y=s-qd=13$ and $\mc M_n^0(qd+y,qd)=\mc M_n^0(337,324).$ Let $B \in \mc M_n^0(337,324)$. 

\quad

\noindent{\bf Digraph of $B^{12}$.} The digraph $\Gamma(B^{12})$ is isomorphic to $\widehat \Gamma^{(12)}$:

$$\widehat \Gamma^{(12)}=C(\langle 12 \cdot {\bf a}(337) \rangle_{n})+\{e_t:t=1,\ldots,12\},$$
where $$e_t=(338-t,26-t), t=1,\ldots,12.$$

\noindent{\bf Congruence modulo $d$.} For $\Pi_j=(j+12{\bf a}_0(h(j))),$ with $h(j)=28+\lfloor \frac{13-j}{12}\rfloor$, the edge $e'_j=e_{\langle 14-j \rangle_{12}}$ connects the vertices in $\Pi_j$ to form a single $27$-cycle. 

For example, the edge $e'_1=e_1=(337,25)$ forms a $27$-cycle by connecting two vertices of $\Pi_1=P(1+12{\bf a}_0(29))$. Similarly,  the edge $e'_2=e_{12}=(326, 14)$ connects vertices in $\Pi_2=P(2+12{\bf a}_0(28))$ to form a $27$-cycle.

Writing $e'_1=(1+12k(1), 1+12l(1))$ gives $k(1)=28$ and $l(1)=2$. Similarly, $e'_2=(2+12k(2),2+12l(2))$ gives $k(2)=27$ and $l(2)=1$.

\noindent{\bf Parts of the Partition.} Recall that $y(j)$ is the number of vertices in $\Pi_j$ before $v_I'(j)=j+12l(j)$. Continuing with $\Pi_1$, from $l(1)=2$ we get $y(1)=2$. Similarly, for $\Pi_2$ we have $l(2)=1$ which gives $y(2)=1$.
More generally, to get all the parts $y(j)$, we write $y$ in terms of $u$ and $\beta$: $y=12u+\beta$, $\beta \in \{0,\ldots,11\}$. This gives us:
\begin{align*}
y(j):=
    \begin{cases}
2 &\text{ for }j \leq 1\\
1 &\text{ for }j >1 .
\end{cases}
\end{align*}

\noindent{\bf Ordering.} From Remark \ref{remark:TI-TIII}, we have:
$${\bf y'} =\left(
\begin{array}{cccccccccccc}
 y(1) & y(12) & y(11) & y(10) & y(9) & y(8) & y(7) & y(6) & y(5) & y(4) & y(3) & y(2) \\
 \end{array}
\right).\\
$$
Thus, ${\bf y'}=\left(
\begin{array}{cccccccccccc}
2 & 1 & 1 & 1 & 1 & 1 & 1 & 1 &1 &1 &1 &1
 \end{array}
\right)$ and the partition class $\PIII(B^{12})$ is given by $\T({\bf y'})$. Note that this uniquely defined partition class for $B^{12}$ is same as the partition class $\PIII(B'^3)$, where $B' \in \mc M_n^0(81,337)$ and $\PIII(B')=\T(\hat {\bf y})$ for $\hat {\bf y}=\npmatrix{4 & 3 & 3 & 3}$  (Item $4$, Example \ref{ex:T3-T3}).

\end{example}

\begin{corollary}\label{cor:TI-TIII}
Let $A \in \mc M_n^0(q,qd+y)$ have the associated partition $\npmatrix{y_1 & \ldots & y_d}\in \PIII(A)$. Then $A=B^d$ for some $B \in \mc M_n^0(qd,qd+y)$ if and only if:
\begin{enumerate}
    \item There exists $u$ such that $y_j \in \{u,u+1\}$. (If all $y_j$ are equal then we say they are all equal to $u$). 
    
    We define $\beta:=|\{j\in \{1\ldots,d\}; y_j=u+1\}|$ and 
    $$y(j):=\begin{cases}
u+1 &\text{ for } j=1,\ldots,\beta \\ 
u  &\text{ for } j=\beta+1,\ldots,d. 
\end{cases}$$

\item The partition class $\PIII(A)$ consists of the parts $y(j)$, $j=1,\ldots,d$, where the part $y(j)$ is followed by the part $y(\langle j-y \rangle_d)$ in $\PIII(A)$.
\end{enumerate}
\end{corollary}

\begin{example}
For $n=337$ we have $\mc K_n(27,337)=\mc K_n(337,324)^{12}$. In this case, the relevant parameters are $q=27$, $s=337$, $d=12$, and $y=13$. By Definition \ref{def:TIIIpartition class} any $12$ non-negative integers that sum up to $13$ form a partition class for some matrix in $\mc M_n^0(27,337)$.
Let $A \in \mc M_n^0(27,337)$ and ${\bf y'}=\npmatrix{y_1 & \ldots & y_{12}} \in \PIII(A)$.

The first part of Corollary \ref{cor:TI-TIII} restricts the elements in ${\bf y'}$ to be in the set $\{u,u+1\}$ for some integer $u$.
This necessary condition leaves us with just one choice for the vector ${\bf y}$ (as defined in Remark \ref{remark:TI-TIII}), which contains the parts of the partition:
$${\bf y}=\left(
\begin{array}{cccccccccccc}
2 & 1 & 1 & 1 & 1 & 1 & 1 & 1 & 1 & 1 & 1 & 1
 \end{array}
\right) . $$
The above ${\bf y}$ gives $\beta=1$ and $u=1$. With these parameters, we get the unique partition class $\PIII(A)$ which is defined in Corollary \ref{cor:TI-TIII}. Therefore $A=B^{12}$ for some $B \in \mc M_n^0(337,324)$ if and only if ${\bf y'}\in \PIII(A)$, where:
\begin{align*}
{\bf y'}={\bf y}&=\left(
\begin{array}{cccccccccccc}
2 & 1 & 1 & 1 & 1 & 1 & 1 & 1 & 1 & 1 & 1 & 1\\
 \end{array}
\right).
\end{align*}
In that case $A=B^{12}=B'^3$ for  $B \in \mc M_n^0(337,324)$ and $B' \in \mc M_n^0(81,337)$, where $\PIII(B')=\T(\hat {\bf y})$ for $\hat{\bf y}=\npmatrix{4 & 3 & 3 &3}.$ 
\end{example}

\noindent
{\bf Acknowledgement. }This publication has emanated from research supported in part by a grant from Science Foundation Ireland under Grant number 18/CRT/6049. For the purpose of Open Access, the author has applied a CC BY public copyright licence to any Author Accepted Manuscript version arising from this submission. The research of the second author is supported in part by NSERC Discovery Grant RGPIN-2019-05408.

\bibliographystyle{plain}
\bibliography{ref}

\end{document}